\documentclass[11pt]{article}
\usepackage[utf8]{inputenc}
\usepackage[T1]{fontenc}

\usepackage{caption}

\usepackage{enumitem}

\usepackage{booktabs}

\usepackage{dcolumn} 

\newcolumntype{d}[1]{D{.}{.}{#1}}

\usepackage{placeins}

\usepackage{amssymb,mathtools}

\usepackage{siunitx}

\usepackage[thmmarks, amsmath]{ntheorem}

\theoremstyle{plain}
\theoremseparator{.}
\newtheorem{theorem}{Theorem}[section]
\newtheorem{lemma}[theorem]{Lemma}
\newtheorem{proposition}[theorem]{Proposition}

\theorembodyfont{\normalfont}

\theoremstyle{nonumberplain}
\theorembodyfont{\normalfont}
\theoremsymbol{\ensuremath{\square}}
\newtheorem{proof}{Proof}

\usepackage{fixme}
\fxsetup{%
  draft,%
  marginface = \tiny,
  english,
  author =, 
}

\usepackage{titlesec}

\titlespacing*{\paragraph}{0pt}{0.5em}{0.5em}

\titleformat{\paragraph}[runin]
{\normalfont\normalsize\bfseries}{\theparagraph}{1em}{}

\newcommand{\bbN}{\mathbb{N}}
\newcommand{\bbZ}{\mathbb{Z}}
\newcommand{\bbR}{\mathbb{R}}

\newcommand{\I}{\mathrm{i}} 

\DeclarePairedDelimiter{\floor}{\lfloor}{\rfloor}

\newcommand{\distTo}[1][]{\xrightarrow[#1]{\:\smash{\textup{d}}\:}}

\newcommand{\fidiTo}[1][]{
    \xrightarrow[#1]{%
        \smash{%
            \mathcal{L}-\textup{f}%
        }
    }
}

\newcommand{\probTo}[1][]{%
    \xrightarrow{\:\p\:}
}

\newcommand{\1}{\mathbf{1}}

\newcommand{\asTo}[1][]{\xrightarrow[#1]{\textup{a.s.}}}

\newcommand{\p}{\mathbb{P}}
\newcommand{\E}{\mathbb{E}}
\DeclareMathOperator{\cov}{Cov}
\DeclareMathOperator{\var}{Var}

\newcommand{\isp}{\,} 
\newcommand{\di}{\textup{d}} 
\newcommand{\Di}{\isp\di}

\DeclareMathOperator*{\argmin}{argmin}

\DeclarePairedDelimiter{\abs}{\lvert}{\rvert}

\DeclarePairedDelimiter{\norm}{\lVert}{\rVert}

\newcommand{\cG}{\mathcal{G}}
\newcommand{\cL}{\mathcal{L}}

\usepackage[%
    linkcolor = blue,%
    citecolor = blue,%
    urlcolor = red,%
    hidelinks,%
    bookmarksnumbered,%
    bookmarksopen,%
]{hyperref}

\begin{document}
\title{A Minimal Contrast Estimator for the Linear Fractional Stable
  Motion}%
\author{%
  Mathias Mørck Ljungdahl%
  \thanks{%
    Department of Mathematics, Aarhus University, 
    ljungdahl@math.au.dk.%
  }%
  \and%
  Mark Podolskij%
  \thanks{%
    Department of Mathematics, Aarhus University,
    mpodolskij@math.au.dk.%
  }%
}

\maketitle

\begin{abstract}
  \noindent In this paper we present an estimator for the
  three-dimensional parameter $(\sigma, \alpha, H)$ of the linear
  fractional stable motion, where $H$ represents the self-similarity
  parameter, and $(\sigma, \alpha)$ are the scaling and stability
  parameters of the driving symmetric L\'evy process $L$.  Our
  approach is based upon a minimal contrast method associated with the
  empirical characteristic function combined with a ratio type
  estimator for the self-similarity parameter $H$. The main result
  investigates the strong consistency and weak limit theorems for the
  resulting estimator.  Furthermore, we propose several ideas to
  obtain feasible confidence regions in various parameter settings.
  Our work is mainly related to \cite{LjunPara,MazuEsti}, in which
  parameter estimation for the linear fractional stable motion and
  related L\'evy moving average processes has been studied.

\bigskip

\noindent\textit{Keywords:} linear fractional processes, Lévy
processes, limit theorems, parametric estimation, bootstrap,
subsampling, self-similarity, low frequency.

\bigskip

\noindent\textit{AMS 2010 subject classifications:} Primary 60G22,
62F12, 62E20; secondary 60E07, 60F05, 60G10.
\end{abstract}

\section{Introduction}

\noindent During the last sixty years fractional stochastic processes
have received a great deal of attention in probability, statistics and
integration theory. One of the most prominent examples of a fractional
model is the (scaled) fractional Brownian motion (fBm), which gained a
lot of popularity in science since the pioneering work of Mandelbrot
and van Ness \cite{MandFrac}. The scaled fBm is the unique zero mean
Gaussian process with stationary increments and self-similarity
property.  As a building block in stochastic models it found numerous
applications in natural and social sciences such as physics, biology
and economics.  From the statistical perspective the scaled fBm is
fully determined by its scale parameter $\sigma >0$ and the
self-similarity parameter (or Hurst index) $H \in (0, 1)$. Nowadays,
the estimation of $(\sigma, H)$ is a well understood problem. We refer
to \cite{DahlEffi} for efficient estimation of the Hurst parameter $H$
in the low frequency setting, and to \cite{BF,CI,IL} for the
estimation of $(\sigma, H)$ in the high frequency setting, among many
others. In more recent papers \cite{BS,LP} statistical inference for
the multifractional Brownian motion has been investigated, which
accounts for the time varying nature of the Hurst parameter.

This paper focuses on another extension of the fBm, the \textit{linear
  fractional stable motion} (lfsm).  The lfsm $(X_t)_{t \geq 0}$ is a
three-parameter statistical model defined by
\begin{equation}
\label{eq:LFM}
X_t = \int_{\bbR} \{ (t - s)_+^{H - 1/\alpha} - (-s)_+^{H - 1/\alpha}\} \Di L_s, 
\end{equation}
where $x_+ = x \vee 0$ denotes the positive part and we set
$x_+^a = 0$ for all $a \in \bbR$, $x \leq 0$. Here
$(L_t)_{t \in \bbR}$ is a symmetric $\alpha$-stable L\'evy process
with $\alpha \in (0, 2)$ and scale parameter $\sigma > 0$, and
$H \in (0,1)$ represents the Hurst parameter. In some sense the lfsm
is a non-Gaussian analogue of fBm.  The process $(X_t)_{t \geq 0}$ has
symmetric $\alpha$-stable marginals, stationary increments and it is
self-similar with parameter $H$. It is well known that the process $X$
has continuous paths when $H - 1/\alpha > 0$, see
e.g. \cite{BenaRoug}. We remark that the class of stationary
increments self-similar processes becomes much larger if we drop the
Gaussianity assumption, cf. \cite{PT,R}, but the lfsm is one of its
most famous representatives due to the ergodicity property.  Linear
fractional stable motions are often used in natural sciences, e.g. in
physics or Internet traffic, where the process under consideration
exhibits stationarity and self-similarity along with heavy tailed
marginals, see e.g. \cite{GLT} for the context of Ethernet and solar
flare modelling.

The limit theory for statistics of lfsm, which is indispensable for
the estimation of the parameter $\xi = (\sigma, \alpha, H)$, turns out
to be of a quite complex nature. First central limit theorems for
partial sums of bounded functions of L\'evy moving average processes,
which in particular include the lfsm, have been discussed in
\cite{PT2} and later extended in \cite{PTA} to certain unbounded
functions. In a more recent work \cite{BassPowe} the authors presented
a rather complete asymptotic theory for power variations of stationary
increments L\'evy moving average processes. Finally, the results of
\cite{BassPowe} have been extended to general functions in
\cite{BassOnLi}, who demonstrated that the weak limit theory crucially
depends on the \textit{Appell rank} of the given function and the
parameters of the model (all functions considered in this paper have
Appell rank 2).  More specifically, they obtained three different
asymptotic regimes, a normal and two stable ones, depending on the
particular setting.  It is this phase transition that depends on the
parameter $(\alpha, H)$ which makes the statistical inference for lfsm
a rather complicated matter.

Since the probabilistic theory for functionals of lfsm was not well
understood until the recent work \cite{BassOnLi,BassPowe}, the
statistical literature on estimation of lfsm is rather scarce. The
articles \cite{AyacLine,PTA} investigate the asymptotic theory for a
wavelet-based estimator of $H$ when $\alpha \in (1, 2)$. In
\cite{BassPowe, SPT} the authors use power variation statistics to
obtain an estimator of $H$, but this method also requires the a priori
knowledge of the lower bound for the stability parameter $\alpha$. The
work \cite{DI} suggested to use negative power variations to get a
consistent estimator of $H$, which applies for any $\alpha \in (0,2)$,
but this article does not contain a central limit theorem for this
method. The paper \cite{MazuEsti} was the first instance, where
estimation of the full parameter $\xi = (\sigma, \alpha, H)$ has been
studied in low and high frequency settings. Their idea is based upon
the identification of $\xi$ through power variation statistics and the
empirical characteristic function evaluated at two different values.

In this paper we aim at extending the approach of \cite{MazuEsti} by
determining the asymptotic theory for the minimal contrast estimator
of the parameter $\xi = (\sigma, \alpha, H)$, which is based upon the
comparison of the empirical characteristic function with its
theoretical counterpart under low frequency sampling. Indeed, the
choice of the two evaluation points for the empirical characteristic
function in \cite{MazuEsti} is rather ad hoc and we will show in the
empirical study that the minimal contrast estimator exhibits better
finite sample properties and robustness in various settings.
Similarly to \cite{MazuEsti}, we will show that the weak limit theory
for our estimator has a normal and a stable regime, and the asymptotic
distribution depends on the interplay between the parameters $\alpha$
and $H$. At this stage we remark that the minimal contrast approach
has been investigated in \cite{LjunPara} in the context of certain
L\'evy moving average models, which do not include the lfsm or its
associated noise process, but only in the asymptotically normal
regime. Another important contribution of our paper is the subsampling
procedure, which provides confidence regions for the parameters of the
model irrespectively of the unknown asymptotic regime.

The article is organized as follows. In
Section~\ref{sec:weak-limit-thm} we introduce the necessary notation
and formulate a new weak limit result related to \cite{MazuEsti},
which is central to their parameter estimation for the linear
fractional stable motion. The aforementioned theorem will be our
starting point, where the aim is to extend the convergence of the
finite dimensional distributions to convergence of integral
functionals appearing in the minimal contrast
method. Section~\ref{sec:results} introduces the estimator and
presents the main results of strong consistency and asymptotic
distribution. Section~\ref{sec:simulation} is devoted to a simulation
study, which tests the finite sample performance of the minimal
contrast estimator. We also discuss the parametric bootstrap and the
subsampling method that are used to construct feasible confidence
regions for the true parameters of the model. All proofs are collected
in Section~\ref{sec:proofs} and all larger tables are in
Section~\ref{sec:tables}.

\section{Notation and recent results}
\label{sec:weak-limit-thm}

We start out with introducing the main notation and statistics of
interest. We consider low frequency observations $X_1, X_2, \ldots, X_n$
from the lfsm $(X_t)_{t \geq 0}$ introduced in \eqref{eq:LFM}. We denote by
$\Delta_{i, k}^{r} X$ ($i, k, r \in \bbN$) the $k$th order increment
of $X$ at stage $i$ and rate $r$, i.e.
\begin{equation*}
\Delta_{i, k}^{r} X = \sum_{j = 0}^{k} (-1)^j \binom{k}{j} X_{i - rj}, \qquad i \geq r k.
\end{equation*}
The order $k$ plays a crucial role in determining the asymptotic
regime for statistics that we introduce below. We let the function
$h_{k, r} : \bbR \to \bbR$ denote the $k$th order increment at rate
$r$ of the kernel in \eqref{eq:LFM}, specifically
\begin{equation*}
h_{k, r}(x) = \sum_{j = 0}^{k} (-1)^{j} \binom{k}{j} (x - rj)_+^{H - 1/\alpha}, \qquad x \in \bbR.
\end{equation*}
We note that
$\Delta_{i, k}^{r} X = \int_{\smash{\bbR}} h_{k, r} (i-s) \Di L_s$.
For less cumbersome notation we drop the index $r$ if $r = 1$, so
$\Delta_{i, k} X \coloneqq \smash{\Delta_{i, k}^{1}} X$ and
$h_k = h_{k, 1}$. Throughout this paper we write
$\theta = (\sigma, \alpha)$ and $\xi = (\sigma, \alpha, H)$. The main
probabilistic tools are statistics of the type
\begin{equation*}
V_n(f, k , r) = \frac{1}{n} \sum_{i = rk}^{n} f(\Delta_{i, k}^{r} X),
\end{equation*}
where $f : \bbR \to \bbR$ is a Borel function satisfying
$\E[|f(\Delta_{rk, k}^{r} X)|]<\infty$. We will specifically focus on
two classes of functions, namely $f_p(x) = |x|^p$ for $p \in (-1,1)$
and $\delta_t(x) = \cos(t x)$ for $t\geq 0$. They correspond to power
variation statistics and the real part of the empirical characteristic
function respectively, and we use the notation
\begin{equation}%
\label{statistics}
\varphi_n(t) = V_n(\delta_t, k, 1) \qquad \text{and} \qquad \psi_n(r) = V_n(f_p, k, r).
\end{equation}
We note that Birkhoff's
ergodic theorem implies the almost sure convergence 
\begin{equation}
\label{eq:empirical-characteric-conv}
\varphi_n(t) \asTo \varphi_{\xi}(t) \coloneqq \exp(-| \sigma \|h_k\|_\alpha t|^{\alpha})
\qquad \text{where} \qquad
\|h_k\|_\alpha^{\alpha} \coloneqq \int_{\mathbb R} |h_k(x)|^{\alpha} \Di x.
\end{equation}
An important coefficient in our context is 
\begin{equation} \label{defbeta}
  \beta = 1 + \alpha (k - H).
\end{equation}
The rate of convergence and the asymptotic distribution of statistics
defined at \eqref{statistics} crucially depend on whether the
condition $k > H + 1/\alpha$ is satisfied or not. Hence, we define the
normalised versions of our statistics as
\begin{equation*}
  W_n^1(r) = \sqrt{n} (\psi_n(r) - r^H m_{p, k}), \quad
  W_n^2(t) = \sqrt{n} (\varphi_n(t) - \varphi_{\xi}(t)) \quad \text{when $k > H + 1 / \alpha$}
\end{equation*}
and
\begin{equation*}
  S_n^1(r) = n^{1 - 1/\beta} (\psi_n(r) - r^H m_{p, k}), \quad
  S_n^2(t) = n^{1 - 1/\beta} (\varphi_n(t) - \varphi_{\xi}(t)) \quad \text{when $k < H + 1 / \alpha$}.
\end{equation*}
Here $m_{p, k} = \E[|\Delta_{k, k} X|^p]$, which is finite for any
$p \in (-1, \alpha)$, and $\varphi_{\xi}$ is given at
\eqref{eq:empirical-characteric-conv}. Note that
$\E[|\Delta_{\smash{r k, k}}^r X|^p] = r^H m_{\smash{p, k}}$, which
explains the centring of $W^1$ and $S^1$.

It turns out that the finite dimensional limit of the statistics
$(W_n^1, W_n^2)$ is Gaussian while the corresponding limit of
$(S_n^1, S_n^2)$ is $\beta$-stable (see Theorem~\ref{thm:weak-limit}
below). We now introduce several notations to describe the limiting
distribution. We start with the Gaussian case. For random variables
$X = \int_{\bbR} g(s) \Di L_s$ and $Y = \int_{\bbR} h(s) \Di L_s$ with
$\|g\|_{\alpha}, \|h\|_{\alpha} < \infty$ we define a dependence
measure $U_{g, h} : \bbR^2 \to \bbR$ as
\begin{equation}
\label{eq:dependence-measure}
\begin{aligned}
  U_{g, h}(u, v) &= \E[\exp(\I (u X + v Y))] - \E[\exp(\I u X)] \E[\exp(\I v Y)]
  \\
                 &= \exp(-\sigma^{\alpha}\| u g + v h\|_{\alpha}^{\alpha}) - \exp(-\sigma^{\alpha}(\| u g \|_{\alpha}^{\alpha} + \| v h\|_{\alpha}^{\alpha})).
\end{aligned}
\end{equation}
Next, for $p \in (-1,1) \setminus \{0\}$, we introduce the constant
\begin{equation}%
\label{ap}%
a_p \coloneqq 
\begin{cases}
  \int_{\bbR} (1 - \cos(y)) |y|^{-1 - p} \Di y: & p \in (0,1)
  \\[1.5 ex]
  \sqrt{2 \pi} \Gamma(-p/2)/ 2^{p + 1/2} \Gamma((p+1)/2): & p \in
  (-1,0),
\end{cases}
\end{equation}
where $\Gamma$ denotes the Gamma function. Now, for each
$t \in \bbR_+ $, we set
\begin{equation}
  \label{Sigmadef}%
\begin{aligned}%
  \Sigma_{11} (g, h) &= a_p^{-2}
                       \int_{\bbR^2} |xy|^{-1-p}U_{g,h}(x,y) \Di x \Di y,
  \\
  \Sigma_{12} (g, h;t) &= a_p^{-1} \int_{\bbR^2}
                         |y|^{-1-p}U_{g,h}(x,t) \Di x
\end{aligned}
\end{equation}
whenever the above integrals are finite. As it has been shown in
\cite{MazuEsti}, the following identities hold for any
$p \in (-1/2, 1/2) \setminus \{0\}$ with $p<\alpha/2$ and
$t \in \bbR_+$:
\begin{equation*}
  \cov(|X|^p, |Y|^p) = \Sigma_{11} (g, h)
  \qquad \text{and} \qquad
  \cov(|X|^p, \exp(\I tY))= \Sigma_{12} (g, h;t).
\end{equation*}
Obviously, the quantities $\Sigma_{11} (g, h)$ and
$\Sigma_{12} (g, h; t)$ will appear in the asymptotic covariance
kernel of the vector $(\psi_n(r), \varphi_n(t))$ in the normal regime.

Now, we introduce the necessary notations for the stable case. First,
we define the functions
\begin{equation}
  \label{eq:Phi-map}
\begin{aligned}
  \Phi_r^{1}(x) &\coloneqq a_p^{-1} \int_{\bbR} (1 - \cos(u x) ) \exp(-|\sigma \|h_{k,r}\|_{\alpha} u|^{\alpha}) |u|^{-1-p} \Di u, \qquad r \in \bbN,
  \\
  \Phi_t^{2}(x) &\coloneqq (\cos(t x) - 1) \exp(-|\sigma \|h_k\|_{\alpha} t|^{\alpha}), \qquad t \geq 0,
\end{aligned}
\end{equation}
and set
\begin{equation*}
  q_{H, \alpha, k}\coloneqq \prod_{i = 0}^{k - 1} (H - 1 / \alpha - i).
\end{equation*}
Next, we introduce the functions $\kappa_1 : \bbN \to \bbR_+$ and
$\kappa_2 : \bbR \to \bbR_-$ via
\begin{equation}
\label{eq:defkappas}
  \kappa_1(r) \coloneqq \frac{\alpha}{\beta} \int_{0}^{\infty} \Phi_r^{1}(q_{H, \alpha, k} z) z^{-1 - \alpha/\beta} \Di z, \quad 
  \kappa_2(t) \coloneqq \frac{\alpha}{\beta} \int_{0}^{\infty} \Phi_t^{2}(  q_{H, \alpha, k} z) z^{-1 - \alpha/\beta} \Di z.
\end{equation}
In the final step we will need to define two L\'evy measures $\nu_1$
on $(\bbR_+)^2$ and $\nu_2$ on $\bbR_+$ that are necessary to
determine the asymptotic distribution of $(S_n^1, S_n^2)$. Let us
denote by $\nu$ the L\'evy measure of the symmetric $\alpha$-stable
L\'evy motion $L$ (i.e. $\nu(\di x)=c(\sigma) |x|^{-1-\alpha} \Di x$)
and define the mappings $\tau_1: \bbR \to (\bbR_+)^2$ and
$\tau_2 : \bbR \to \bbR_+$ via
\begin{equation*}
  \tau_1(x) = |x|^{\alpha/\beta} (\kappa_1(1), \kappa_1(2)), \qquad \tau_2(x) = |x|^{\alpha/\beta}.
\end{equation*}
Then, for any Borel sets $A_1 \subseteq (\bbR_+)^2$ and
$A_2 \subseteq \bbR_{+}$ bounded away from $(0, 0)$ and $0$,
respectively, we introduce
\begin{equation}
\label{defnu}
\nu_l(A_l)\coloneqq \nu(\tau_l^{-1}(A_l)), \qquad l=1,2.  
\end{equation} 
In the weak limit theorem below we write $Z^n \fidiTo Z$ to denote the
convergence of finite dimensional distributions, i.e. the convergence
in distribution
\begin{equation*}
  (Z^n_{t_1}, \ldots, Z^n_{t_d}) \distTo (Z_{t_1}, \ldots, Z_{t_d})
\end{equation*}
for any $d \in \bbN$ and $t_j \in \bbR_+$. The following theorem is
key for statistical applications.

\begin{theorem}
\label{thm:weak-limit}
Assume that either $p \in (-1/2, 0)$ or $p \in (0, 1/2)$ together with
$p < \alpha / 2$.
\begin{enumerate}[label = (\roman{*})]
\item\label{it:thm:weak-limit:1} If $k > H + 1/\alpha$ then as
  $n \to \infty$
  \begin{equation*}
    (W_n^1(1), W_n^1(2), W_n^2(t)) \fidiTo (W^1_1, W^1_2, W_t),
  \end{equation*}
  where $W^1 = (W^1_1, W^1_2)$ is a centered $2$-dimensional normal
  distribution and $(W_t)_{t \geq 0}$ is a centred Gaussian process
  with
  \begin{align*}
    \cov(W^1_j, W^1_{j'}) &=  \sum_{l \in \bbZ} \Sigma_{11}(h_{k,j}, h_{k,j'}(\cdot+l)) \qquad j,j'=1,2,
    \\
    \cov(W^1_j, W_t) &=  \sum_{l \in \bbZ} \Sigma_{12}(h_{k,j}, h_{k}(\cdot+l);t) \qquad j=1,2,~ t \in \bbR_+,
    \\
    \cov(W_s, W_{t}) &=  \frac{1}{2} \sum_{l \in \bbZ} \bigl(U_{h_{k}, h_{k}(\cdot + l)} (s,t)
                       + U_{h_{k},-h_{k}(\cdot+l)} (s, t) \bigr) \qquad s,t \in \bbR,
  \end{align*}
  where the quantity $\Sigma_{ij}$ has been introduce at
  \eqref{Sigmadef}. Moreover, the Gaussian process $W$ exhibits a
  modification (denoted again by $W$), which is locally H\"older
  continuous of any order smaller than $\alpha/4$.

\item\label{it:thm:weak-limit:2} If $k < H + 1/\alpha$ then as
  $n \to \infty$
  \begin{equation*}
    (S_n^1(1), S_n^1(2), S_n^2(t)) \fidiTo (S^1_1, S^1_2, \kappa_2(t)S).
  \end{equation*}
  where $S^1 = (S^1_1, S^1_2)$ is a $\beta$-stable random vector with
  L\'evy measure $\nu_1$ independent of the totally right skewed
  $\beta$-stable random variable $S$ with L\'evy measure $\nu_2$, and
  $\nu_1,\nu_2$ have been defined in \eqref{defnu}.
\end{enumerate}
\end{theorem}

\noindent
The finite dimensional asymptotic distribution demonstrated in
Theorem~\ref{thm:weak-limit} is a direct consequence of
\cite[Theorem~2.2]{MazuEsti}, which even contains a more general
multivariate result. However, the smoothness property of the limiting
Gaussian process $W$ and the particular form of the limit of $S_n^2$
have not been investigated in \cite{MazuEsti}. Both properties are
crucial for the statistical analysis of the minimal contrast
estimator.

We observe that from a statistical perspective it is
more favourable to use
Theorem~\ref{thm:weak-limit}\ref{it:thm:weak-limit:1} to estimate the
parameter $\xi = (\sigma, \alpha, H)$, since the convergence rate
$\sqrt{n}$ in \ref{it:thm:weak-limit:1} is faster than the rate
$n^{1-1/\beta}$ in \ref{it:thm:weak-limit:2}. However, the phase
transition happens at the point $k=H+1/\alpha$, which depends on
unknown parameters $\alpha$ and $H$. This
poses major difficulties in statistical applications and we will
address this issue in the forthcoming discussion.

\section{Main results}
\label{sec:results}

In this section we describe our minimal contrast approach and present
the corresponding asymptotic theory. Before stating our the main
result we define a power variation based estimator of the parameter
$H \in (0, 1)$. Since the increments of the process $(X_t)_{t \geq 0}$
are strongly ergodic (cf. \cite{CambErgo}), we deduce by Birkhoff's
ergodic theorem the almost sure convergence
\begin{equation*}
  \psi_n(r) = \frac 1n \sum_{i = rk}^{n} |\Delta_{i, k}^r X|^p
  \asTo \E[|\Delta_{rk, k}^r X|^p] = r^{pH} m_{p, k}
\end{equation*}
for any $p \in (-1,\alpha)$. In
particular, we have that
\begin{equation*}
R_n(p, k) \coloneqq \frac{\psi_n(2)}{\psi_n(1)} \asTo 2^{p H},
\end{equation*}
consequently yielding a consistent estimator $H_n(p, k)$ of
$H$ as
\begin{equation}
\label{eq:estimator-H}
H_n(p, k) = \tfrac{1}{p} \log_2(R_n(p, k)) \asTo H
\end{equation}
for any $p \in (-1,\alpha)$. The idea of using negative powers
$p \in (-1,0)$ to estimate $H$, which has been proposed in \cite{DI}
and applied in \cite{MazuEsti}, has the obvious advantage that it does
not require knowledge of the parameter
$\alpha$. From Theorem~\ref{thm:weak-limit}\ref{it:thm:weak-limit:1}
and the $\delta$-method applied to the function
$v_p(x, y) = \frac{1}{p}(\log_2(x) - \log_2(y))$ we immediately deduce
the convergence
\begin{equation}
\label{eq:limit-Hn-normal}
\bigl( \sqrt{n}(H_n(p, k) - H),  W_n^2(t) \bigr) \fidiTo (M_1, W_t)
\end{equation}
for $k>H+1/\alpha$, where $M_1$ is a centred Gaussian random variable. 
Similarly, when $k < H + 1/\alpha$, we deduce the convergence
\begin{equation}
\label{eq:limit-H_n-stable}
\bigl( n^{1 - 1/\beta}(H_n(p, k) - H) ,  S_n^2(t)\bigr) \fidiTo (M_2, \kappa_2(t) S)
\end{equation}
from Theorem~\ref{thm:weak-limit}\ref{it:thm:weak-limit:2}.  

We will now introduce the minimal contrast estimator of the parameter
$\theta = (\sigma, \alpha)$. Let $w \in \cL^1(\bbR_+)$ denote a
positive weight function. Define for $r > 1$ the norm
\begin{equation*}
\| h\|_{w, r} = \Bigl(\int_{0}^{\infty} |h(t)|^r w(t) \Di t\Bigr)^{1/r},
\end{equation*}
where $h : \bbR_+ \to \bbR$ is a  Borel function. Denote by
$\cL_w^r(\bbR_+)$ the space of functions $h$ with $\|h\|_{w, r} < \infty$. Let
$(\theta_0, H_0) = (\sigma_0, \alpha_0, H_0) \in (0, \infty) \times
(0, 2) \times (0, 1)$ be  the true parameter of the model
\eqref{eq:LFM} and consider an open neighbourhood
$\Theta_0 \subseteq (0, \infty) \times (0, 2)$ around $\theta_0$
bounded away from $(0, 0)$. Define the map
$F : \cL_w^2(\bbR_+) \times (0, 1) \times \Theta_0 \to \bbR$ as
\begin{equation}
  \label{Fdef}
F(\varphi, H, \theta) = \|\varphi - \varphi_{\theta, H}\|_{w, 2}^2,
\end{equation}
where $\varphi_{\theta, H}$ is the limit introduced at
\eqref{eq:empirical-characteric-conv}. We define the minimal contrast
estimator $\theta_n$ of $\theta_0$ as
\begin{equation}
\label{eq:minimal-contrast-estimator}
\theta_n \in \argmin_{\theta \in \Theta_0} F(\varphi_n, H_n(p, k), \theta).
\end{equation}
We remark that by e.g. \cite[Theorem~2.17]{StinSome} it is possible to
choose $\theta_n$ universally measurable with respect to the
underlying probability space. The
joint estimator of $(\theta_0, H_0)$ is then given as
\begin{equation*}
  \xi_n = (\theta_n, H_n(p,k))^{\top}.
\end{equation*}
Below we denote by $\nabla_\theta$ the gradient with respect to the
parameter $\theta$ and similarly by $\nabla_\theta^2$ the Hessian. We
further write $\partial_z f$ for the partial derivative of $f$ with
respect to $z \in \{\sigma, \alpha, H\}$. The main theoretical result
of this paper is the following theorem.

\begin{theorem}
\label{thm:main-result}
Suppose $\xi_0 = (\theta_0, H_0)$ is the true parameter of the linear
fractional stable motion $(X_t)_{t \geq 0}$ at \eqref{eq:LFM} and that
the weight function $w \in \cL^1(\bbR_+)$ is continuous.
\begin{enumerate}[label = (\roman{*})]
\item\label{it:thm:main-result:1} $\xi_n \asTo \xi_0$ as
  $n \to \infty$.

\item\label{it:thm:main-result:2} If $k > H + 1/\alpha$ then
  \begin{align*}
    \sqrt{n} (\xi_n - \xi_0)
 &\distTo
   \Bigl[-2 \nabla_{\theta}^2 F(\varphi_{\xi_0}, \xi_0)^{-1} \Bigl(\int_{0}^{\infty} W_t \nabla_{\theta} \varphi_{\xi_0}(t) w(t) \Di t
    \\
 &\quad + \partial_H \nabla_{\theta} F(\varphi_{\xi_0}, \xi_0) M_1 \Bigr), M_1 \Bigr]^{\top}.
  \end{align*}

\item\label{it:thm:main-result:3} If $k < H + 1/\alpha$
  then
  \begin{align*}
    n^{1 - 1/\beta} (\xi_n - \xi_0)
    &\distTo
      \Bigl[-2 \nabla_{\theta}^2 F(\varphi_{\xi_0}, \xi_0)^{-1} \Bigl(S \int_{0}^{\infty} \kappa_2(t)  \nabla_{\theta} \varphi_{\xi_0}(t) w(t) \Di t
    \\
    &\quad + \partial_H \nabla_{\theta} F(\varphi_{\xi_0}, \xi_0) M_2\Bigr), M_2\Bigr]^{\top}.
  \end{align*}
\end{enumerate}
\end{theorem}

\noindent In principle, the statement of Theorem~\ref{thm:main-result}
follows from \eqref{eq:limit-Hn-normal}, \eqref{eq:limit-H_n-stable}
and an application of the implicit function theorem. For general
infinite dimensional functionals of our statistics we would usually
need to show tightness of the process $(W_n^2(t))_{t \geq 0}$ (or
$(S_n^2(t))_{t \geq 0}$), which is by far not a trivial
issue. However, in the particular setting of integral functionals, it
suffices to show a weaker condition that is displayed in
Proposition~\ref{prop:fidi-int-conv}. Indeed, this is the key step of
the proof.

\section{Simulations, parametric bootstrap and subsampling}
\label{sec:simulation}

The theoretical results of Theorem~\ref{thm:main-result} are far from
easy to apply in practice. There are a number of issues, which need to
be addressed. First of all, since the parameters $H_0$ and $\alpha_0$
are unknown, we do not know whether we are in the regime of
Theorem~\ref{thm:main-result}\ref{it:thm:main-result:2}
or~\ref{it:thm:main-result:3}. Furthermore, even if we could determine
whether the condition $k>H_0+1/\alpha_0$ holds or not, the exact
computation or a reliable numerical simulation of the quantities
defined in \eqref{Sigmadef} and \eqref{eq:Phi-map} seems to be out of
reach. Below we will propose two methods to overcome these
problems. In the setting where the lfsm $(X_t)_{t\geq 0}$ is
continuous, which corresponds to the condition $H_0-1/\alpha_0>0$, we
will see that it suffices to choose $k=2$ to end up in the normal
regime of Theorem~\ref{thm:main-result}\ref{it:thm:main-result:2}. The
confidence regions are then constructed using the parametric bootstrap
approach. In the general setting we propose a novel subsampling method
which, in some sense, automatically adapts to the unknown limiting
distribution.

For comparison reasons we include the estimation of the parameter
$H_0$ using $H_n(p, k)$ defined at \eqref{eq:estimator-H}, even though
its properties have already been studied in \cite{MazuEsti}. Moreover,
we pick our weight function in the class of Gaussian kernels:
\begin{equation*}
  w_\nu(t) = \exp(-\tfrac{t^2}{2 \nu^2}) \qquad (t, \nu > 0).
\end{equation*}
In particular we can use Gauss-Hermite quadrature, see
\cite{SteeGaus}, to estimate the integral
\begin{equation*}
  \|\varphi_n - \varphi_{\xi} \|_{w, 2}^2
  = \int_{0}^{\infty} (\varphi_n(t) - \varphi_{\xi}(t))^2 w_\nu(t) \Di t.
\end{equation*}
This procedure is based on a number of weights, unless otherwise
stated we pick 12~weights. We mentioned that it is possible to choose
other weight functions and standard numerical procedures exits for
these. We restrict our simulation study to three different values of
$\nu \in \{0.05, 0.1, 1\}$ for the bootstrap method in
Section~\ref{sec:bootstrap}. Additionally, while the theoretical
characteristic function $\varphi_\xi$ has an explicit form it depends
on the norm $\norm{h_k}_\alpha$ which is not readily computable, hence
needs to be approximated.

To produce the simulation study we generate observations from the lfsm
using the recent \texttt{R} package \texttt{rlfsm}, which implements
an algorithm based on \cite{StoeSimu}; this package already includes
an implementation of the minimal contrast estimator. To compute the
estimator a minimisation
\begin{equation*}
  \argmin_{\sigma, \alpha} \int_{0}^{\infty} (\varphi_n(t) - \varphi_{\sigma, \alpha, H_n(p, k)}(t))^2 w_\nu(t) \Di t
\end{equation*}
has to be carried out. For this purpose we use \cite{NeldASim}, which in
particular entails picking a starting point for the algorithm. For
$\sigma$ no immediate choice exists, so we simply pick $\sigma = 2$,
while $\alpha = 1$ seems obvious. 

\subsection{Empirical bias and variance}

In this section we will check the bias and variance performance of our
minimal contrast estimator. First, we consider the empirical bias and
standard deviation, which are summarized in
Tables~\ref{tab:abs-bias-est-1000} and~\ref{tab:std-est-1000} for
$n = \num{1000}$ and Tables~\ref{tab:abs-bias-est}
and~\ref{tab:std-est} for $n = \num{10000}$. These are based on Monte
Carlo simulation with at least $\num{1000}$ repetitions. We fix the
parameters $k = 2$, $p = -0.4$, $\nu = 0.1$ and $\sigma_0 = 0.3$ and
perform the estimation procedure for various values of $\alpha$ and
$H$.

At this stage we recall that due to
Theorem~\ref{thm:main-result}\ref{it:thm:main-result:3} we obtain a
slower rate of convergence when $H_0 + 1/\alpha_0 > 2$; the stable
regime is indicated in bold in
Tables~\ref{tab:abs-bias-est-1000}--\ref{tab:std-est}. This explains a
rather bad estimation performance for $\alpha_0=0.4$. The effect is
specifically pronounced for the parameter $\sigma$, which has the
worst performance when $\alpha_0=0.4$. This observation is in line
with the findings of \cite{MazuEsti}, who concluded that the scale
parameter $\sigma$ is the hardest to estimate in practice. Also the
starting point $\sigma = 2$ of the minimisation algorithm, which is
not close to $\sigma_0 = 0.3$, might have a negative effect on the
performance. The estimation performance for values $\alpha_0 > 0.4$ is
quite satisfactory for all parameters, improving from $n = \num{1000}$
to $n = \num{10000}$. We remark the superior performance of our method
around the value $\alpha=1$, which is explained by the fact that
$\alpha = 1$ is the starting point of the minimisation procedure.

For comparison, we display the bias and standard deviation of our
estimator for $k = 1$ based on $n = \num{10000}$ observations in
Tables~\ref{tab:abs-bias-est-1} and~\ref{tab:std-est-1}, where the
stable regime is again highlighted in bold. We see a better finite
sample performance for $\alpha_0 = 0.4$, but in most other cases we
observe a larger bias and standard deviation compared to $k = 2$. This
is explained by slower rates of convergence in the setting of the
stable regime and $k = 1$.

We will now compare the minimal contrast estimator with the estimator
proposed in \cite{MazuEsti}. To recall the latter estimator we observe
the following identities due to \eqref{eq:empirical-characteric-conv}:
\begin{equation*}
  \sigma = \frac{(-\log\varphi_\xi(t_1;k))^{1/\alpha}}{t_1 \norm{h_k}_\alpha}, \quad
  \alpha = \frac{\log \abs{ \log \varphi_\xi(t_2; k)} - \log \abs{\log \varphi_\xi(t_1; k)}}{\log t_2 - \log t_1},
  \quad \xi = (\sigma, \alpha, H)
\end{equation*}
for fixed values $0 < t_1 < t_2$. Since $h_k$ depends on $\alpha$ and
$H$ we immediately obtain a function $G : \bbR^3 \to \bbR^2$ such that
\begin{equation*}
  (\sigma, \alpha) = G(\varphi_\xi(t_1; k), \varphi_\xi(t_2; k), H), \qquad \xi = (\sigma, \alpha, H).
\end{equation*}
Hence an estimator for $(\sigma, \alpha, H)$ is obtained by insertion
of the empirical characteristic function and $H_n(p; k)$ from
\eqref{eq:estimator-H}:
\begin{equation}
  \label{eq:old-estimator}
  (\widetilde{\sigma}_{\textrm{low}}, \widetilde{\alpha}_{\textrm{low}}, H_n(p; k))
  = G(\varphi_n(t_1; k), \varphi_n(t_2; k), H_n(p; k)).
\end{equation}%
%
%
The first comparison of the estimators is between
Tables~\ref{tab:abs-bias-est} and~\ref{tab:std-est} for the minimal
contrast estimator with Tables~\ref{tab:abs-bias-est-old}
and~\ref{tab:std-est-old}, all of which are based on at least
$\num{1000}$ Monte Carlo repetitions and the same parameter
choices. We see that the minimal contrast estimator outperforms the
old estimator for values $\alpha \leq 0.8$, where the latter is
completely unreliable in most cases. But for larger values of $\alpha$
the old estimator might be slightly better. However, the results for
the estimation of the scale parameter $\sigma$ in the case
$\alpha = 0.4$ and $H_0 \in \{0.2, 0.4\}$ are hard to interpret, since
$\widetilde{\sigma}_{\textrm{low}}$ often delivers values that are
indistinguishable from $0$ yielding a small variance and relatively
small bias for $\sigma_0 = 0.3$.

Another point is the instability of the old
estimator. Table~\ref{tab:failure-rate-old}, which
is based on at least $\num{1200}$ simulations,  shows the rate at which
the estimators fail to return a value $\alpha \in (0,2)$. In this
regard the minimal contrast estimator is far superior in most cases
and actually in the case $k = 1$ the estimator almost never fails, although
we dispense with the simulation results. We remark that in
theory the minimal contrast estimator should never return $\alpha$'s
not in the interval $[0, 2]$. However, we apply the minimisation
procedure from \cite{NeldASim} that does not allow constrained
optimisation. One could instead use e.g. the procedure from
\cite{ByrdALim}; we choose the first method as it is hailed as very
robust.

An additional advantage of the minimal contrast estimator is that it
allows incorporation of a priori knowledge of the parameters $\sigma$
and $\alpha$, using the weight function but also the starting point
for the minimisation algorithm.

\subsection{Bootstrap inference in the continuous case}
\label{sec:bootstrap}

In this section we only consider the continuous case, which
corresponds to the setting $H_0 - 1/\alpha_0 > 0$. In this setup
$H_0 \in (1/2, 1)$ and $\alpha_0 \in (1, 2)$ must hold. We can in
particular choose $k = 2$ to ensure that
Theorem~\ref{thm:main-result}\ref{it:thm:main-result:2} applies, thus
yielding the faster convergence rate $\sqrt{n}$. We are interested in
obtaining feasible confidence regions for all parameters, but, as we
mentioned earlier, the computation or reliable numerical approximation
of the asymptotic variance in
Theorem~\ref{thm:main-result}\ref{it:thm:main-result:2} is out of
reach. Instead we propose the following parametric bootstrap procedure
to estimate the confidence regions for the true parameters.

\begin{enumerate}[label = (\arabic*)]
\item Compute the minimal contrast $\xi_n$ estimator for given
  observations $X_1, \ldots, X_n$.

\item\label{it:bootstrap:2} Generate new samples
  $X_1^j, \ldots, X_n^j$ for $j = 1, \ldots, N$ using the parameter
  $\xi_n$.

\item Compute new estimators $\xi_n^{(j)}$ from the samples generated
  in \ref{it:bootstrap:2} for each $j = 1, \ldots, N$.

\item Calculate the empirical variance $\hat{\Sigma}_n$ of $\xi_n$
  based on the estimators $\xi_n^{(1)}, \ldots, \xi^{(N)}_n$.

\item For each parameter construct 95\%-confidence regions based on
  the relation
  \begin{equation*}
    \sqrt{n} (\xi_n - \xi_0) \approx \mathcal N(0, n \hat{\Sigma}_n).
  \end{equation*}
\end{enumerate}

\noindent
To test this we repeated the above procedure for \num{200} Monte Carlo
simulations with $N =
\num{200}$. Tables~\ref{tab:bootstrap-hitting-prob-1}
and~\ref{tab:bootstrap-hitting-prob-2} report acceptance rates for
$\sqrt{n} (\xi_n - \xi_0)$ for an approximate \SI{95}{\%}-confidence
interval. We observe a good performance for all estimators with the
exception of $n = \num{1000}$ for $\sigma$ with $\nu = 0.05, 0.1$ and
$\alpha$ for $\nu = 0.05$ in
Table~\ref{tab:bootstrap-hitting-prob-1}. The estimator is fairly
stable under changes in $\nu$ in this parameter regime, but it should
be mentioned that a smaller $\nu$-value does lead to a larger failure
rate, we dispense with the numerics.

\begin{table}[htpb!]
  \centering
  \caption{Acceptance rates for the true parameter
    $(\sigma_0, \alpha_0, H_0) = (0.3, 1.8, 0.8)$ and power $p = 0.4$}
  \label{tab:bootstrap-hitting-prob-1}
  \begin{tabular}{p{7mm} *{3}{S[table-format = 2.1, round-precision
        = 1]}}
    \multicolumn{4}{c}{$\nu = 0.1$}
    \\
    \toprule
    $n$ & \multicolumn{1}{c}{$\sigma$} & \multicolumn{1}{c}{$\alpha$} & \multicolumn{1}{c}{$H$}
    \\
    \midrule
    \num{1000} & 69.15888 & 95.79439 & 96.26168
    \\
    \num{2500} & 96.60194 & 98.54369 & 95.63107
    \\
    \num{5000} & 99.02439 & 99.5122 & 96.09756
    \\
    \bottomrule
  \end{tabular}%
\quad%
  \begin{tabular}{p{7mm} S[table-format = 3, round-precision
        = 2] *{2}{S[table-format = 2.1, round-precision
        = 1]}}
    \multicolumn{4}{c}{$\nu = 1$}
    \\
    \toprule
    $n$ & \multicolumn{1}{c}{$\sigma$} & \multicolumn{1}{c}{$\alpha$} & \multicolumn{1}{c}{$H$}
    \\
    \midrule
    \num{1000} & 100 & 95.54731 & 94.80519
    \\
    \num{2500} & 100 & 96.2963 & 93.05556
    \\
    \num{5000} & 100 & 93.54839  & 93.54839
    \\
    \bottomrule
  \end{tabular}%
  \quad%
  \begin{tabular}{p{7mm} *{2}{S[table-format = 3.1, round-precision =
        1]} S[table-format = 2.1, round-precision = 1]}
    \multicolumn{4}{c}{$\nu = 0.05$}
    \\
    \toprule
    $n$ & \multicolumn{1}{c}{$\sigma$} & \multicolumn{1}{c}{$\alpha$} & \multicolumn{1}{c}{$H$}
    \\
      \midrule
      \num{1000} & 81.73077 & 44.23077 & 99.51923
      \\
      \num{2500} & 99.57265 & 97.86325 & 97.86325
      \\
      \num{5000} & 100 & 99.54545 & 96.81818
      \\
    \bottomrule
  \end{tabular}%
\end{table}

\begin{table}[htpb!]
  \centering
  \caption{Acceptance rates for the true parameter
    $(\sigma_0, \alpha_0, H_0) = (0.3, 1.3, 0.8)$ and power $p = 0.4$}
  \label{tab:bootstrap-hitting-prob-2}
  \begin{tabular}{p{7mm} *{3}{S[table-format = 2.1, round-precision
        = 1]}}
    \multicolumn{4}{c}{$\nu = 0.1$}
    \\
    \toprule
    $n$ & \multicolumn{1}{c}{$\sigma$} & \multicolumn{1}{c}{$\alpha$} & \multicolumn{1}{c}{$H$}
    \\
    \midrule
    \num{1000} & 87.09677 & 95.39171 & 93.08756
    \\
    \num{2500} & 90.09901 & 91.58416 & 97.52475
    \\
    \num{5000} & 91.62562 & 91.62562 & 95.07389
    \\
    \bottomrule
  \end{tabular}
  \quad%
  \begin{tabular}{p{7mm} *{3}{S[table-format = 2.1, round-precision
        = 1]}}
    \multicolumn{4}{c}{$\nu = 1$}
    \\
    \toprule
    $n$ & \multicolumn{1}{c}{$\sigma$} & \multicolumn{1}{c}{$\alpha$} & \multicolumn{1}{c}{$H$}
    \\
    \midrule
    \num{1000} & 93.93939 & 95.671 & 93.93939
    \\
    \num{2500} & 95.71429 & 96.66667 & 94.28571
    \\
    \num{5000} & 93.54839 & 93.08756 & 93.54839
    \\
    \bottomrule
  \end{tabular}%
  \quad%
    \begin{tabular}{p{7mm} *{3}{S[table-format = 2.1, round-precision
        = 1]}}
      \multicolumn{4}{c}{$\nu = 0.05$}
      \\
      \toprule
      $n$ & \multicolumn{1}{c}{$\sigma$} & \multicolumn{1}{c}{$\alpha$} & \multicolumn{1}{c}{$H$}
      \\
      \midrule
      \num{1000} & 91.8552 & 98.64253 & 95.9276
      \\
      \num{2500} & 89.7561 & 96.58537 & 95.12195
      \\
      \num{5000} & 91.42857 & 94.28571 & 97.14286
      \\
    \bottomrule
  \end{tabular}%
\end{table}

\FloatBarrier

\subsection{Subsampling method in the general case}
\label{sec:subsampling}

In contrast to the continuous setting $H_0 - 1/\alpha_0 > 0$, there
exists no a priori choice of $k$ in the general case, which ensures
the asymptotically normal regime of
Theorem~\ref{thm:main-result}\ref{it:thm:main-result:2}. This problem
was tackled in \cite{MazuEsti} via the following two stage approach.
In the first step they obtained a preliminary estimator
$\alpha_n^0(t_1,t_2)$ of $\alpha_0$ using \eqref{eq:old-estimator} for
$k = 1$, $t_1 = 1$ and $t_2 = 2$. In the second step they defined the
random number
\begin{equation}
\label{eq:k-hat}
  \hat{k} = 2 + \lfloor \alpha_n^0(t_1, t_2)^{-1} \rfloor,
\end{equation}
and computed the estimator
$(\widetilde{\sigma}_{\textrm{low}},
\widetilde{\alpha}_{\textrm{low}}, H_n(p, k))$ based on
$k=\hat{k}$. They showed that the resulting estimator is
$\sqrt n$-consistent and derived the associated weak limit
theory. However, this approach does not completely solve the original
problem, since they obtained four different convergence regimes
according to whether $1 > H_0 + 1/\alpha_0$ or not, and whether
$\alpha_0^{-1} \in \bbN$ or not.

Nevertheless, we apply their idea to propose a new subsampling method
to determine feasible confidence regions for the parameters of the
model. For our procedure it is crucial that the convergence rate is
known explicitly and the weak convergence of the involved statistics
is insured. We proceed as follows:

\begin{enumerate}[label = (\arabic*)]
\item\label{it:subsamp:1} Given observations $X_1, \ldots, X_n$ 
  compute $\hat{k}$ from \eqref{eq:k-hat} and construct the minimal
  contrast estimator $\xi_n = (\sigma_n, \alpha_n, H_n(p, \hat{k}))$.

\item Split $X_1, \ldots, X_n$ into $L$ groups such that the $l$th
  group contains the data $(X_{(l-1)n/L + j})_{j = 1}^{n/L}$ ($n/L$ is
  assumed to be an integer). For each $l = 1, \ldots, L$ calculate
  $\hat{k}_l$ from \eqref{eq:k-hat}.
  
\item\label{it:subsamp:3} For each $l = 1, \ldots, L$ construct the
  minimal contrast estimators
  $(\sigma_n^{\smash{(l)}}, \alpha_n^{\smash{(l)}})$ and
  $H_n^{\smash{(l)}}(p, \hat{k}_l)$ based on the $l$th group. For the
  estimation of $(\sigma, \alpha)$ use $H_n(p, \hat{k})$ from
  \ref{it:subsamp:1} as plug-in.

\item Compute the $97.5\%$ and $2.5\%$ quantiles for each of the
  distribution functions
\end{enumerate}
\begin{equation*}
  \frac{1}{L} \sum_{l = 1}^{L} \1_{\bigl\{\sqrt{\frac{n}{L}}(\sigma_n^{(l)} - \sigma_n) \leq x\bigr\}},
  \quad
  \frac{1}{L} \sum_{l = 1}^{L} \1_{\bigl\{\sqrt{\frac{n}{L}}(\alpha_n^{(l)} - \alpha_n) \leq x\bigr\}},
  \quad     \frac{1}{L} \sum_{l = 1}^{L} \1_{\bigl\{\sqrt{\frac{n}{L}}(H_n^{\smash{(l)}}(p, \hat{k}_l) - H_n(p, \hat{k})) \leq x\bigr\}}.
\end{equation*}

\noindent Let us explain the intuition behind the proposed subsampling
procedure. First of all, similarly to the theory developed in
\cite{MazuEsti}, the minimal contrast estimator $\xi_n$ obtained
through a two step method described in the beginning of the section
leads to four different limit regimes for $\sqrt{n} (\xi_n - \xi_0)$
(although we leave out the theoretical derivation here). Using this
knowledge we may conclude that, for each $l=1,\ldots, L$,
$\sqrt{n/L} (\xi_n^{\smash{(l)}} - \xi_0)$ has the same (unknown)
asymptotic distribution as the statistic $\sqrt{n} (\xi_n - \xi_0)$ as
long as $n/L \to \infty$. Since the true parameter $\xi_0$ is unknown,
we use its approximation $\xi_n$, which has a much faster rate of
convergence than $\sqrt{n/L}$ when $L\to \infty$. Finally, the
statistics constructed on different blocks are asymptotically
independent, which follows along the lines of the proofs in
\cite{MazuEsti}. Hence, the law of large numbers implies that the
proposed subsampling statistics converge to the unknown true
asymptotic distributions when $L\to \infty$ and $n/L \to \infty$.

In Tables~\ref{tab:subsamp-1} and~\ref{tab:subsamp-2} we report the
empirical $95\%$-confidence regions for the parameters of the model
using the subsampling approach. We perform $500$ Monte Carlo
simulations and choose $n = 12.5 \times L^2$.

\begin{table}[htbp!]
  \centering
  \caption{Acceptance rates for the true parameter
    $(\sigma_0, \alpha_0, H_0) = (0.3, 0.8, 0.8)$. Here $p = -0.4$ and
    $\nu = 0.1$}
  \label{tab:subsamp-1}
  \begin{tabular}{p{0.8cm} p{1.2cm} *{3}{S[round-precision = 2, table-format = 2.2]}}
    \toprule
    $L$ & $n/L$ & \multicolumn{1}{c}{$\sigma$} & \multicolumn{1}{c}{$\alpha$} & \multicolumn{1}{c}{$H$}
    \\
    \midrule
    80 & \num{1000} & 90.65421 & 94.39252 & 89.71963
    \\
    100 & \num{1250} & 87.71593 & 94.24184 & 89.25144
    \\
    \bottomrule
  \end{tabular}
\end{table}

\begin{table}[htbp!]
  \centering
  \caption{Acceptance rates for the true parameter
    $(\sigma_0, \alpha_0, H_0) = (0.3, 1.8, 0.8)$. Here $p = -0.4$ and
    $\nu = 0.1$}
  \label{tab:subsamp-2}
 \begin{tabular}{p{0.8cm} p{1.2cm} *{3}{S[round-precision = 2, table-format = 2.2]}}
    \toprule
    $L$ & $n/L$ & \multicolumn{1}{c}{$\sigma$} & \multicolumn{1}{c}{$\alpha$} & \multicolumn{1}{c}{$H$}
    \\
    \midrule
    80 & \num{1000} & 60.69789 & 67.30946 & 92.19467
    \\
    100 & \num{1250} & 68.88218 & 74.92447 & 94.56193
    \\
    \bottomrule
  \end{tabular}
\end{table}

Table~\ref{tab:subsamp-1} shows a satisfactory performance for all
estimators, while the results of Table~\ref{tab:subsamp-2} are quite
unreliable for the parameters $\sigma$ and $\alpha$. The reason for
the latter finding is the suboptimal finite sample performance of the
estimators in the case of $(\alpha, H) = (1.8, 0.8)$, which is
displayed in Tables~\ref{tab:abs-bias-est-1000}
and~\ref{tab:std-est-1000}.

We conclude this section by remarking the rather satisfactory
performance of our estimator in the continuous setting
$H_0 - 1 / \alpha_0 > 0$. On the other hand, when using
the subsampling method in the general setting, a further careful
tuning seems to be required. In particular, the choice of the weight
function $w$ and the group number $L$ plays an important role in
estimator's performance. We leave this study for future research.

\section{Proofs}
\label{sec:proofs}

We denote by $C$ a finite, positive constant which may differ from
line to line. Moreover, any important dependence on other constants
warrants a subscript. To simplify notations we set $H_n = H_n(p,k)$. 

\subsection{Proof of
  Theorem~\ref{thm:weak-limit}\ref{it:thm:weak-limit:1}}
\label{sec:gaussian-case}

As we mentioned earlier, the convergence of finite dimensional
distributions has been shown in \cite[Theorem~2.2]{MazuEsti}, and thus
we only need to prove the smoothness property of the limit $W$. We
recall the definition of the quantity $U_{g,h}$ at
\eqref{eq:dependence-measure} and start with the following lemma.

\begin{lemma}[{{\cite[Eqs.~(3.4)--(3.6)]{PTA}}}]
\label{lem:dependence-measure}
Let $g, h \in \cL^\alpha(\bbR)$. Then for any $u, v \in \bbR$
\begin{align*}
  |U_{g,h}(u, v)| &\leq 2 | u v|^{\alpha/2} \int_0^{\infty} | g(x) h(x) |^{\alpha/2} \Di x
  \\
                  &\times \exp\Bigl( -2|u v|^{\alpha/2} \bigl(\|g\|_{\alpha}^{\alpha} \|h\|_{\alpha}^{\alpha} - \int_{0}^{\infty} |g(x)h(x)|^{\alpha/2} \Di x \bigr) \Bigr),
  \\
  |U_{g, h}(u, v)| &\leq 2 |u v|^{\alpha /2} \int_{0}^{\infty} |g(x) h(x)|^{\alpha/2} \Di x
  \\
                  &\times \exp\bigl(-\bigl(\|u g\|_{\alpha}^{\alpha/2} - \|v h\|_{\alpha}^{\alpha/2}\bigr)^2 \bigr).
  \end{align*}
  In particular, it holds that
  $|U_{g, h}(u, v)| \leq 2 |u v|^{\alpha/2} \int_{0}^{\infty} |g(x)
  h(x)|^{\alpha/2} \Di x$.
\end{lemma}

\noindent Next, we define for each $l \in \bbZ$ 
\begin{equation*}
  \rho_l = \int_{0}^{\infty} |h_k(x) h_k(x + l)|^{\alpha /2} \Di x
\end{equation*}
and recall the following lemma from \cite{MazuEsti}.

\begin{lemma}[{{\cite[Lemma~6.2]{MazuEsti}}}]
\label{lem:rho-bound}
If $k > H + 1/\alpha$ and $l > k$ then
\begin{equation*}
\rho_l \leq l^{(\alpha(H - k) - 1) / 2}.
\end{equation*}
\end{lemma}
To prove that the process $W$ is locally H\"older continuous of any
order smaller than $\alpha/4$, we use Kolmogorov's
criterion.
Since $W$ is a Gaussian process it suffices to prove that for each
$T > 0$ there exists a constant $C_T \geq 0$ such that
\begin{equation}
\label{eq:holder-inequality-W2}
  \E[(W_t - W_s)^2] \leq C_T |t - s|^{\alpha/2} \qquad \text{for all $s, t \in [0, T]$.}
\end{equation}
This is performed in a similar fashion as in
\cite[Section~4.1]{LjunPara}. First, we reduce the problem. Using
$\cos(t x) = (\exp(\I t x) + \exp(-\I t x)) / 2$ and the symmetry of
the distribution of $X$, we observe the identity
\begin{align*}
  \cov(\cos(t\Delta_{j,k} X), \cos(s\Delta_{j+l,k} X))
  = \frac{1}{2} \Bigl( U_{h_k, -h_k(l+\cdot)}(t, s) + U_{h_k, h_k(l+\cdot)}(t, s) \Bigr). 
\end{align*}
In the following we focus on the first term in the above decomposition
(the second term is treated similarly). More specifically, we will
show the inequality \eqref{eq:holder-inequality-W2} for the quantity
$\overline{r}(t, s)$, which is given as
\begin{align*}
  \overline{r}(t, s) &= \sum_{l \in \bbZ} \overline{r}_l(t, s) \quad \text{where}
  \\
  \overline{r}_l(t, s) &=U_{h_k, -h_k(l+\cdot)}(t, s)
  \\
                     &= \exp(-\|t h_k - s h_k(l + \cdot)\|_\alpha^\alpha) - \exp(-(t^\alpha + s^\alpha)\|h_k\|_\alpha^\alpha).
\end{align*}
Moreover, since
$\overline{r}(t, t) + \overline{r}(s, s) - 2 \overline{r}(t, s) \leq
|\overline{r}(t, t) - \overline{r}(t, s)| + |\overline{r}(s, s) -
\overline{r}(t, s)|$ it is by symmetry enough to prove that
\begin{equation*}
  |\overline{r}(t, t) - \overline{r}(t, s)| \leq C_T |t - s|^{\alpha/2}
  \qquad \text{for all $s, t \in [0, T]$.}
\end{equation*}
For $l \in \bbZ$ decompose now as follows:
\begin{align*}
  \MoveEqLeft[0] \overline{r}_l(t, t) - \overline{r}_l(t, s)
  = \exp(-2 t^\alpha \|h_k\|_\alpha^\alpha) \Bigl[\exp(-\|t(h_k - h_k(\cdot + l))\|_\alpha^\alpha + 2 t^\alpha \|h_k\|_\alpha^\alpha ) - 1\Bigr]
  \\
  &- \exp(-(t^\alpha + s^\alpha)\|h_k\|_\alpha^\alpha) \Bigl[ \exp(-\|t h_k - s h_k(l + \cdot)\|_\alpha^\alpha + (t^\alpha + s^\alpha)\|h_k\|_\alpha^\alpha) - 1 \Bigr]
  \\
  &= \Bigl[\exp(-2 t^\alpha \|h_k\|_\alpha^\alpha) - \exp(-(t^\alpha + s^\alpha) \|h_k\|_\alpha^\alpha)\Bigr]
  \\
  &\times \Bigl[\exp(-\|t(h_k - h_k(l + \cdot))\|_\alpha^\alpha + 2 t^\alpha \|h_k\|_\alpha^\alpha) - 1\Bigr] + \exp(-(t^\alpha + s^\alpha)\|h_k\|_\alpha^\alpha)
  \\
  &\times \Bigl[\exp(-\|t(h_k - h_k(l + \cdot))\|_\alpha^\alpha + 2 t^\alpha \|h_k\|_\alpha^\alpha) - \exp(-\|t h_k - s h_k(l + \cdot)\|_\alpha^\alpha + (t^\alpha + s^\alpha) \|h_k\|_\alpha^\alpha)\Bigr]
  \\
  &\eqqcolon \overline{r}_l^{(1)}(t, s) + \overline{r}_l^{(2)}(t, s).
\end{align*}
Applying the second inequality of Lemma~\ref{lem:dependence-measure}
and the mean value theorem we deduce the estimate:
\begin{equation}
  \label{eq:bound-r1}
  |\overline{r}^{(1)}(t, s)| \leq C_T \rho_l |t^\alpha - s^\alpha| \leq C_T \rho_l |t - s|^{\alpha/2}
  \qquad \text{for all $s, t \in [0, T]$.}
\end{equation}
Again by the mean value theorem we find that
\begin{equation*}
|\overline{r}^{(2)}(t, s)| \leq C_T \Bigl| \|t h_k - s h_k(l + \cdot)\|_\alpha^\alpha - \|t (h_k - h_k(l + \cdot))\|_\alpha^\alpha + (t^\alpha - s^\alpha) \|h_k\|_\alpha^\alpha \Bigr|.
\end{equation*}
The last term can be rewritten as
\begin{align*}
  \MoveEqLeft\|t h_k - s h_k(l + \cdot)\|_\alpha^\alpha - \|t(h_k - h_k(l + \cdot))\|_\alpha^\alpha + (t^\alpha - s^\alpha) \|h_k\|_\alpha^\alpha
  \\
&= \int_{0}^{\infty} |t h_k(x) - s h_k(x + l)|^\alpha - |t(h_k(x) - h_k(l + x))|^\alpha + (t^\alpha - s^\alpha) |h_k(x + l)|^\alpha \Di x.
\end{align*}
As $ \alpha \in (0, 2)$ we have
$|x^\alpha - y^\alpha| \leq |x^2 - y^2|^{\alpha/2}$ for all
$x, y \geq 0$. In particular
\begin{align*}
  \bigl||t h_k(x) - s h_k(x + l)|^\alpha - |t(h_k(x) - h_k(x + l))|^\alpha \bigr|
  &\leq C_T |t - s|^{\alpha/2}
  \\
  &\times (|h_k(x + l)|^\alpha + |h_k(x) h_k(x + l)|^{\alpha/2}).
\end{align*}
It then follows that
\begin{equation}
  \label{eq:bound-r2}
  |\overline{r}_l^{(2)}(t, s)| \leq C_T |t - s|^{\alpha/2} (\rho_l + \mu_l)
  \qquad \text{for all $s, t \in [0, T]$}
\end{equation}
where $\mu_l$ is the quantity defined as
\begin{equation*}
  \mu_l = \int_{0}^{\infty} |h_k(x + l)|^\alpha \Di x.
\end{equation*}
It remains to prove that
\begin{equation}
  \label{eq:summability-rho-mu}
  \sum_{l \in \bbZ} \rho_l < \infty \qquad \text{and} \qquad \sum_{l \in \bbZ} \mu_l < \infty.
\end{equation}
The first claim is a direct consequence of
Lemma~\ref{lem:rho-bound}. The second convergence can equivalently be
formulated as
\begin{equation*}
  \sum_{l = 1}^{\infty} l \int_{l}^{l + 1} |h_k(x)|^\alpha \Di x < \infty.
\end{equation*}
Recall that $|h_k(x)| \leq C |x|^{H - 1/\alpha - k}$ for large $x$,
hence
\begin{equation*}
  \sum_{l = 1}^{\infty} l \int_{l}^{l + 1} |h_k(x)|^\alpha \Di x
  \leq C \sum_{l = 1}^{\infty} l \int_{l}^{l + 1} x^{\alpha(H - k) - 1} \Di x
  \leq C \sum_{l = 1}^{\infty} l^{\alpha(H - k)} < \infty,
\end{equation*}
where we used the assumption $k > H + 1/\alpha$. Combining
\eqref{eq:bound-r1} and~\eqref{eq:bound-r2} with
\eqref{eq:summability-rho-mu} we can conclude
\eqref{eq:holder-inequality-W2}, and hence the proof of
Theorem~\ref{thm:weak-limit}\ref{it:thm:weak-limit:1} is complete.

\subsection{Proof of
  Theorem~\ref{thm:weak-limit}\ref{it:thm:weak-limit:2} }
\label{stable-case}

We recall that the asymptotic distribution of the vector
$(S_n^1(1), S_n^1(2))$ and its asymptotic independence of $S_n^2(t)$
have been shown in \cite[Theorem~2.2]{MazuEsti}. Hence, we only need
to determine the functional form of the limit of the statistic
$S_n^2(t)$.

In the following we will recall a number of estimates and
decompositions from \cite[Theorem~2.2]{MazuEsti}, which will be also
helpful in the proof of
Theorem~\ref{thm:main-result}\ref{it:thm:main-result:3}. We start out
with a series of estimates on the function $\Phi_t^2$ given at
\eqref{eq:Phi-map}, but for a general scale parameter $\eta > 0$. Let
$\Phi_{t, \eta}$ denote the function
\begin{equation*}
\Phi_{t, \eta}(x) = \E[\cos(t (Y + x))] - \E[\cos(t Y)] \qquad x \in \bbR,
\end{equation*}
where $Y$ is an S$\alpha$S distributed random variable with scale
parameter $\eta$. We obviously have the representation
\begin{equation}
\label{eq:Phit-representation}
\Phi_{t, \eta}(x) = (\cos(x t) - 1) \exp(-\abs{\eta t}^{\alpha}).
\end{equation}
The next lemma gives some estimates on the function $\Phi_{t, \eta}$. 

\begin{lemma}
\label{lem:bounds-Phi}
For $\eta > 0$ set $g_{\eta}(t) = \exp(-\abs{\eta t}^{\alpha})$ and
let $\Phi_{t, \eta}^{(v)}(x)$ denote the $v$th derivative at
$x \in \bbR$. Then there exists a constant $C > 0$ such that for all
$t \geq 0$ it holds that
\begin{enumerate}[label = (\roman*)]
\item\label{it:lem:bounds-Phi:1}
  $\abs{\Phi_{t, \eta}^{(v)}(x)} \leq C t^v g_{\eta}(t)$ for all
  $x \in \bbR$ and $v \in \{0, 1, 2\}$.

\item\label{it:lem:bounds-Phi:2}
  $\abs{\Phi_{t, \eta}(x)} \leq g_{\eta}(t) (1 \wedge \abs{x t}^2)$.

\item\label{it:lem:bounds-Phi:3}
  $\abs{\Phi_{t, \eta}(x) - \Phi_{t, \eta}(y)} \leq t^2 g_{\eta}(t)
  ((1 \wedge \abs{x} + 1 \wedge \abs{y}) \abs{x - y} \1_{\{\abs{x -
        y} \leq 1\}} + \1_{\{\abs{x - y} > 1\}})$.

\item\label{it:lem:bounds-Phi:4} For any $x, y > 0$ and $a \in \bbR$
  then
  \begin{equation*}
    F(a, x, y) \coloneqq \abs[\bigg]{\int_{0}^{y} \int_{0}^{x} \Phi_{t, \eta}^{(v)}(a + u + v) \Di u \Di v} \leq C g_{\eta}(t) (t + 1)^2 (1 \wedge x) (1 \wedge y).
  \end{equation*}
\end{enumerate}
\end{lemma}

\begin{proof}
  \noindent\textbf{\ref{it:lem:bounds-Phi:1}:} This follows directly
  from \eqref{eq:Phit-representation}.

 \paragraph{\ref{it:lem:bounds-Phi:2}:} This is straightforward using
 the standard inequality $1 - \cos(y) \leq y^2$.

\paragraph{\ref{it:lem:bounds-Phi:3}:} \ref{it:lem:bounds-Phi:1}
  implies that
  $\abs{\Phi_{t, \eta}^{(1)}(x)} \leq t^2 g_{\eta}(t) (1 \wedge
  \abs{x})$ and note that
\begin{equation*}
\abs{\Phi_{t, \eta}(x) - \Phi_{t, \eta}(y)} = \abs[\Big]{\int_{y}^{x} \Phi_{t, \eta}^{(1)}(u) \Di u}.
\end{equation*}
If $\abs{x - y} > 1$ we simply bound the latter by $t^2
g_{\eta}(t)$. If $\abs{x - y} \leq 1$, then by the mean value theorem
there exists a number $s$ with $|x - s|\leq |x-y|$ such that
\begin{equation*}
\abs[\Big]{\int_{y}^{x} \Phi_{t, \eta}^{(1)}(u) \Di u} = \abs{\Phi_{t, \eta}^{(1)}(s)} \abs{x - y} .
\end{equation*}
Observe then
\begin{equation*}
  \abs{\Phi_{t, \eta}^{(1)}(s)} \leq t^2 g_{\eta}(t) (1 \wedge \abs{s})
  \leq t^2 g_{\eta}(t) (1 \wedge (\abs{x} + \abs{y})).
\end{equation*}
This completes the proof of the inequality.

\paragraph{\ref{it:lem:bounds-Phi:4}:} Let $a$, $x$ and $y$ be
given. Observe that
\begin{align*}
  \int_{0}^{y} \int_{0}^{x} \Phi_{t, \eta}^{(2)}(a + u + v) \Di u \Di v
  &= \int_{0}^{y} \Phi_{t, \eta}^{(1)}(a + x + v) - \Phi_{t, \eta}^{(1)}(a + v) \Di v
  \\
  &= \Phi_{t, \eta}(a + x + y) - \Phi_{t, \eta}(a + y) - (\Phi_{t, \eta}(a + x) - \Phi_{t, \eta}(a)).
\end{align*}
The last equality implies that $F(a, x, y) \leq C g_{\eta}(t)$. The
first equality implies that $F(a, x, y) \leq C g_{\eta}(t) t
y$. Reversing the order of integration we get a similar expression as
the first equality with $x$ replaced by $y$. Hence,
$F(a, x, y) \leq C g_{\eta}(t) t x$. Lastly, using
\ref{it:lem:bounds-Phi:1} on the first integral yields
$F(a, x, y) \leq C g_{\eta}(t) t^2 x y$. Splitting into the four cases
completes the proof.
\end{proof}

\noindent We will consider the asymptotic decomposition of the
statistic $S_n^2(t)$ given in \cite[Section~5]{BassOnLi} (see also
\cite{BassPowe,MazuEsti}). We set
\begin{equation*}
S_n^2(t) = n^{-1/\beta} \sum_{i = k}^{n} (\cos(t \Delta_{i, k} X) - \varphi_{\xi}(t)) \eqqcolon n^{-1/\beta} \sum_{i = k}^{n} V_i(t).
\end{equation*}
Define for each $s \geq 0$ we define the $\sigma$-algebras
\begin{equation*}
  \cG_s = \sigma (L_v - L_u : v, u \leq s) \qquad \text{and} \qquad \cG_s^1 = \sigma (L_v - L_u : s \leq v, u \leq  s + 1).
\end{equation*}
We also set  for all $n \geq k$, $i \in \{k, \ldots, n\}$ and
$t \geq 0$
\begin{equation*}
  R_i(t) = \sum_{j = 1}^{\infty} \zeta_{i, j}(t) \qquad \text{and} \qquad
  Q_i(t) = \sum_{j = 1}^{\infty} \E[V_i(t) \mid \cG_{i - j}^1],
\end{equation*}
where
\begin{equation*}
  \zeta_{i, j}(t) = \E[V_i(t) \mid \cG_{i - j + 1}]
  - \E[ V_i(t) \mid \cG_{i - j}] - \E[V_i(t) \mid \cG_{i - j}^1].
\end{equation*}
Then the following decomposition holds:
\begin{equation}
\label{eq:decomp-Sn}
S_n^2(t) = n^{-1/\beta} \sum_{i = k}^{n} R_i(t) + \Bigl( n^{-1/\beta} \sum_{i = k}^{n} Q_i(t) - \overline{S}_n(t) \Bigr) +  \overline{S}_n(t),
\end{equation}
where
\begin{equation}\label{Phibar}
  \begin{aligned}[b]
  \overline{S}_n(t) ={}& n^{-1/\beta} \sum_{i = k}^{n} (\overline{\Phi}_t(L_{i + 1} - L_{i}) - \E[\overline{\Phi}_t(L_{i + 1} - L_{i})]),
  \\
  \overline{\Phi}_t(x) \coloneqq{}& \sum_{i = 1}^{\infty} \Phi_t^2
                         (h_k(i) x).
\end{aligned}
\end{equation}
It turns out that the first two terms in \eqref{eq:decomp-Sn} are
negligible while $\overline{S}_n$ is the dominating term. More
specifically, we can use similar arguments as in
\cite[Eq.~(5.22)]{BassPowe} and deduce the following proposition from
Lemma~\ref{lem:bounds-Phi}.

\begin{proposition}
\label{prop:R-moment}
For any $\varepsilon > 0$ there exists a constant $C > 0$ such that
for all $n \in \bbN$
\begin{equation*}
  \sup_{t \geq 0} \E\biggl[\Bigl(n^{-1/\beta}\sum_{i = k}^{n} R_i(t) \Bigr)^2\biggr]
  \leq C n^{2(2 - \beta - 1/\beta) + \varepsilon}.
\end{equation*}
\end{proposition}

\noindent Using the inequality $2 - x - 1/x < 0$ for all $x > 1$ on
$\beta > 1$ it follows by picking $\varepsilon > 0$ small enough that
the first term in \eqref{eq:decomp-Sn} is asymptotically
negligible. Decomposing the second term and using arguments as in the
equations (5.30), (5.31) and~(5.38) in \cite{BassPowe} we obtain the
following result.
\begin{proposition}
\label{prop:R-moment2}
  For any $\varepsilon > 0$ there exist an $r > 1$, an
  $r' > \beta \vee r$ and a constant $C > 0$ such that for all $n$ in
  $\bbN$ 
  \begin{equation*}
    \sup_{t \geq 0} \E\Bigl[\Bigl| n^{-1/\beta} \sum_{i = k}^{n} Q_i(t) - \overline{S}_n(t) \Bigr|^r\Bigr]
    \leq C \bigl(n^{r(\varepsilon + 2 - \beta - 1/\beta)} + n^{\frac{r}{r'}(1 - r'/\beta)}\bigr)
  \end{equation*}
\end{proposition}
Using again the inequality $2 - x - 1/x < 0$ for all $x > 1$, it follows
immediately that the second term in \eqref{eq:decomp-Sn} is
asymptotically negligible. Hence, $\overline{S}_n(t)$ is
asymptotically equivalent to the statistic $S_n^2(t)$, and it suffices
to analyse its finite dimensional distribution.

Consider $t_1, \ldots, t_d \in \bbR_+$. We will now recall the
limiting distribution of the vector
$(\overline{S}_n(t_1), \ldots, \overline{S}_n(t_d))$. Observing the
definition \eqref{Phibar}, we deduce the uniform convergence
\begin{equation}
\label{eq:asympt-Phi}
\sup_{t \geq 0} \bigl||x|^{-\alpha/\beta} \overline{\Phi}_t(x) - \kappa_2(t) \bigr| \to 0
\qquad \text{as $x \to -\infty$,}
\end{equation}
where $\kappa_2$ has been introduced at \eqref{eq:defkappas}. Indeed
by substituting $u = (|x|/z)^{\alpha/\beta}$ we have that
\begin{align*}
  \MoveEqLeft \sup_{t \geq 0}  \bigl||x|^{-\alpha/\beta} \overline{\Phi}_t(x) - \kappa_2(t) \bigr|
  \\
  &= \sup_{t \geq 0} \Bigl| |x|^{-\alpha/\beta} \int_{0}^{\infty} \Phi_t^2(h_k(\lfloor u \rfloor + 1) |x|) \Di u
    - \int_{0}^{\infty} \Phi_t^2(q_{H, \alpha, k} z) z^{-1 - \alpha/\beta} \Di u \Bigr|
  \\
  &= \frac{\alpha}{\beta} \sup_{t \geq 0} \Bigl| \int_{0}^{\infty} \Phi_t^2(h_k(\lfloor (|x|/z)^{\alpha/\beta} \rfloor + 1) |x|) z^{-1 - \alpha/\beta} \Di u
    - \int_{0}^{\infty} \Phi_t^2(q_{H, \alpha, k} z) z^{-1 - \alpha/\beta} \Di u \Bigr|
  \\
  &\leq \frac{\alpha}{\beta} \int_{0}^{\infty} \sup_{t \geq 0} \bigl| \Phi_t^2(h_k(\lfloor (|x|/z)^{\alpha/\beta} \rfloor + 1) x) - \Phi_t^2(q_{H, \alpha, k} z) \bigr| z^{-1 - \alpha/\beta} \Di u.
\end{align*}
By Lemma~\ref{lem:bounds-Phi}\ref{it:lem:bounds-Phi:3} the integrand
vanishes pointwise in $z$ as $x \to -\infty$ due to the asymptotics
\begin{equation}
\label{eq:asymp-hk}
h_k(x) \sim q_{H, \alpha, k} x^{-\beta/\alpha} \qquad \text{as $x \to \infty$.}
\end{equation}
Due to
Lebesgue's dominated convergence theorem it is enough to bound the
integrand uniformly in $x < - 1$. By the triangle inequality it is
enough to treat each $\Phi_t^2$-term separately. For the first term
Lemma~\ref{lem:bounds-Phi}\ref{it:lem:bounds-Phi:2} implies that
\begin{equation*}
  \sup_{t \geq 0} |\Phi_t^2(h_k(\lfloor (|x|/z)^{\alpha/\beta} \rfloor + 1) |x|)| z^{-1 - \alpha/\beta}
  \leq C \bigl(1 \wedge |h_k(\lfloor (|x|/z)^{\alpha/\beta} \rfloor + 1) x|^2 \bigr) z^{-1 - \alpha/\beta}.
\end{equation*}
For large $z$, say $z > 1$, the latter is bounded by the integrable
function $z^{-1 - \alpha/\beta} \1_{\{z > 1\}}$. For $z \in (0, 1]$ we
deduce by \eqref{eq:asymp-hk}
\begin{equation*}
  |h_k(\lfloor (|x|/z)^{\alpha/\beta} \rfloor + 1) x|^2 z^{-1 - \alpha/\beta}
  \leq C x^2 (\lfloor (|x|/z)^{\alpha/\beta} \rfloor + 1)^{-2\beta/\alpha} z^{-1 - \alpha/\beta}
  \leq C z^{1 - \alpha/\beta},
\end{equation*}
where we used that
$(|x|/z)^{\alpha/\beta} \leq \lfloor (|x|/z)^{\alpha/\beta} \rfloor +
1$. Recalling \eqref{defbeta}, we deduce that $\alpha /\beta < 2$ and
an integrable bound is obtained. The second $\Phi_t^2$-term is treated
similarly. Hence, we have \eqref{eq:asympt-Phi}.

Define the map
$\tau_{t_1,\ldots, t_d} : \bbR \to (\bbR_-)^d$ as
\begin{equation*}
  \tau_{t_1,\ldots, t_d}(x) = |x|^{\alpha/\beta} (\kappa_2(t_1), \ldots, \kappa_2(t_d)).
\end{equation*}
Following \cite[Lemma~6.6]{BassOnLi} the limit of the vector  $(\overline{S}_n(t_1), \ldots, 
\overline{S}_n(t_d))$ is determined by the L\'evy measure  
\begin{equation*}
\nu_{t_1,\ldots, t_d}(A) \coloneqq  \nu(\tau_{t_1,\ldots, t_d}^{-1}(A)),
\end{equation*}
where $A \subseteq (\bbR_-)^d$ is a Borel set and $\nu$ is the L\'evy
measure of $L$. But $\nu_{t_1,\ldots, t_d}$ is also the L\'evy measure
of the vector $(\kappa_2(t_1), \ldots, \kappa_2(t_d))S$, where the
random variable $S$ has been introduced in
Theorem~\ref{thm:weak-limit}\ref{it:thm:weak-limit:2}. This implies
the desired result.

\subsection{Finite dimensional convergence and integral functionals}
\label{sec:dist-conv-integrals}

Let $(Y^n)_{n\geq 1}$ and $Y$ be stochastic processes index by $\bbR_+$
with paths in $\cL^1(\bbR_+)$. We will  give simple sufficient conditions
for when the implication
\begin{equation}
\label{eq:dist-conv-integrals}
Y^n \fidiTo Y \quad\implies\quad \int_{0}^{\infty} Y^n_u \Di u \distTo \int_{0}^{\infty} Y_u \Di u
\end{equation}
holds true. Such a result is obviously required to obtain
Theorem~\ref{thm:main-result} from Theorem~\ref{thm:weak-limit}.
Before we state these conditions we remark that the question has
already been studied in the literature. As an example Theorem~22 in
Appendix~I of \cite{IbraStat} gives two sufficient conditions for
\eqref{eq:dist-conv-integrals} to hold, but the second condition is a
Hölder type criteria, which is not easily verifiable in our
setting. Moreover, the theorem only deals with integration over
bounded sets. The article \cite{CremOnWe} studies this question in
general, but the conditions of e.g. Lemma~1 therein are too abstract
even though we are in the case of a finite measure (the one induced by
the weight function $w$). What can be deduced from \cite{CremOnWe} is
that some kind of uniform integrability (with respect to the product
measure) is sufficient for \eqref{eq:dist-conv-integrals}.

To formulate the lemma, define for each $n, m, l \in \bbN$ the
intermediate random variables
\begin{equation*}
  X_{n, m, l} = \int_{0}^{l} Y^n_{\floor{u m}/m} \Di u \qquad\text{and}\qquad
  X_{n, l} = \int_{0}^{l} Y^n_u \Di u.
\end{equation*}

\begin{proposition}
\label{prop:fidi-int-conv}
Suppose that $(Y^n)_{n \geq 1}$ and $Y$ are continuous stochastic
processes and assume that the following conditions hold:
\begin{equation}
  \label{eq:integral-conv-conds}
  \lim_{l \to \infty} \limsup_{n \to \infty} \int_{l}^{\infty} \E[|Y^n_u|] \Di u =  0, \quad
  \lim_{m \to \infty} \limsup_{n \to \infty} \p(|X_{n, m, l} - X_{n, l}| \geq \varepsilon) = 0,
\end{equation}
for all $l, \varepsilon > 0$. Then
\eqref{eq:dist-conv-integrals} holds.
\end{proposition}

\begin{proof}
  Note that for each $n,m,l \in \bbN$ we have the decomposition
\begin{equation*}
  \int_{0}^{\infty} Y_u^n \Di u = X_{n, m, l} + (X_{n, l} - X_{n, m, l}) + \int_{l}^{\infty} Y_u^n \Di u.
\end{equation*}
As $Y^n \fidiTo Y$ we deduce the weak convergence
\begin{equation*}
  X_{n, m, l} \distTo[n \to \infty] Y_{m, l} \coloneqq \int_{0}^{l} Y_{\floor{u m}/m} \Di u
  \qquad \text{for each $m \in \bbN$.}
\end{equation*}
The continuity of $Y$ implies immediately that
$Y_{m, l} \asTo \int_{0}^{l} Y_u \Di u$ as $m \to \infty$. The assumptions in
\eqref{eq:integral-conv-conds} then imply the convergence
$\int_{0}^{\infty} Y_u^n \Di u \distTo \int_{0}^{\infty} Y_u \Di u$.
\end{proof}

\subsection{Proof of
  Theorem~\ref{thm:main-result}\ref{it:thm:main-result:1}
  and~\ref{it:thm:main-result:2}}

The strong consistency result of
Theorem~\ref{thm:main-result}\ref{it:thm:main-result:1} is an
immediate consequence of \eqref{eq:estimator-H} and
$\|\varphi_n - \varphi_{\xi_0}\|_{w, 2} \asTo 0$, where the latter
follows from \eqref{eq:empirical-characteric-conv} and the dominated
convergence theorem. Hence, we are left to proving
Theorem~\ref{thm:main-result}\ref{it:thm:main-result:2}. 

Recall the definition of the function
$F : \cL_w^2(\bbR_+) \times (0, 1) \times \Theta_0 \to \bbR$ at \eqref{Fdef}, 
where $\Theta_0 \subseteq (0, \infty) \times (0, 2)$ is an open neighbourhood 
of $(\sigma_0, \alpha_0)$ bounded away from $(0, 0)$. Now, the minimal
contrast estimator at \eqref{eq:minimal-contrast-estimator} can be
obtained using the criteria
\begin{equation*}
\nabla_{\theta} F(\psi, H, \theta) = 0,
\end{equation*}
which is satisfied at $(\varphi_{\xi_0}, \xi_0)$. Denote by
$\zeta(\psi, H)$ an element of $(0, 1) \times \Theta_0$ such that
\begin{equation*}
  \nabla_{\theta} F(\psi, H, \zeta(\psi, H)) = 0.
\end{equation*}
To determine the derivative of $\zeta$ we will need the infinite dimensional  implicit
function theorem, which we briefly repeat.

Consider three Banach spaces $(E_i, \|\cdot \|_i)$, $i = 1, 2, 3$, and
open subsets $U_i \subseteq E_i$, $i = 1, 2$. Let
$f : U_1 \times U_2 \to E_3$ be a Fréchet differentiable map. For a
point $(e_1, e_2) \in U_1 \times U_2$ and a direction
$(h_1, h_2) \in E_1 \times E_2$ we denote by $D_{e_1, e_2}^k f(h_k)$,
$k = 1, 2$, the Fréchet derivative of $f$ at the point $(e_1, e_2)$ in
the direction $h_k \in E_k$. Assume that
$(e_1^0, e_2^0) \in U_1 \times U_2$ satisfies the equation
$f(e_1^0, e_2^0) = 0$ and that the map
$D_{\smash{e_1^0, e_2^0}}^2 f : E_2 \to E_3$ is continuous and
invertible. Then there exists open sets $V_1 \subseteq U_1$ and
$V_2 \subseteq U_2$ with $(e_1^0, e_2^0) \in V_1 \times V_2$ and a
bijective function $G : V_1 \to V_2$ such that
\begin{equation*}
f(e_1, e_2) = 0 \quad \iff \quad G(e_1) = e_2.
\end{equation*}
Moreover, $G$ is Fréchet differentiable with derivative
\begin{equation}
\label{eq:derivative-of-implicit}
D_{e_1} G (h) = - (D^2_{e_1, G(e_1)} f)^{-1} (D^1_{e_1, G(e_1)} f(h)).
\end{equation}
We will adapt this to our setup, which corresponds to
$U_1 = \cL_w^r(\bbR_+) \times (0, 1)$, for a $r > 1$,
$E_2 = \Theta_0$, $E_3 = \bbR$ and $f = \nabla_{\theta} F$. A
straightforward calculation shows that the map
$\theta \mapsto \nabla^2_\theta F(\varphi, H, \theta)$ is
differentiable with derivative at $(\varphi_{\xi_0}, \xi_0)$
represented by the Hessian
\begin{equation*}
  D^2_{\varphi_{\xi_0}, \xi_0} \nabla_{\theta} F
  = \nabla_{\theta}^2 F(\varphi_{\xi_0}, \xi_0)
  = 2 \Bigl(\int_{0}^{\infty} \partial_{\theta_i} \varphi_{\xi_0}(t) \partial_{\theta_j} \varphi_{\xi_0}(t) w(t) \Di t \Bigr)_{i,j = 1,2}.
\end{equation*}
The linear independence of the maps
$\partial_{\theta_1} \varphi_{\xi_0}$ and
$\partial_{\theta_2} \varphi_{\xi_0}$ immediately shows the
invertibility of the Hessian. Moreover, standard theory for
convergence in $\cL^r(\bbR_+)$, $r > 1$, shows that the map
$(\varphi, H) \mapsto \nabla_\theta^2 F(\varphi, H, \cdot)$ is
continuous, which is needed to assert that $\nabla_\theta F$ is $C^1$.

The determination of the remaining derivative
$D^1_{\smash{\varphi, \xi}} \nabla_\theta F$ for a point
$(\varphi, \xi) \in \cL^r_w(\bbR_+) \times (0, 1) \times \Theta_0$ is
slightly more involved. It is given by its two components
$D^1 = (D^{1, 1}, D^{1, 2})$ corresponding to the partial
derivatives. Indeed, $D^{1, 1}_{\smash{\varphi, \xi}} \nabla_\theta F$
is the derivative with respect to the functional coordinate
$\varphi \in \cL^r_w(\bbR_+)$ and
$D^{1, 2}_{\smash{\varphi, \xi}} \nabla_\theta F$ the derivative with
respect to the Hurst parameter $H \in (0, 1)$, where
$\xi = (H, \alpha, \sigma)$. It is easily seen that
\begin{equation*}
  D^{1, 1}_{\varphi, \xi} \nabla_{\theta} F(h) = D^{1, 1}_{\xi} \nabla_{\theta} F(h)
  = - 2\int_{0}^{\infty} h(t) \nabla_{\theta} \varphi_{\xi}(t) w(t) \Di t, \qquad h \in \cL^r_w(\bbR_+).
\end{equation*}
An application of Hölder's inequality proves the continuity of the
linear map $\xi \mapsto D^{1, 1}_\xi \nabla_{\theta} F$. The second
partial derivative at $(\varphi, \xi)$ is the linear map represented
by the two dimensional vector
\begin{equation*}
  D^{1, 2}_{\varphi, \xi} \nabla_{\theta} F
  = 2 \int_{0}^{\infty} \partial_H\varphi_{\xi}(t) \nabla_{\theta} \varphi_{\xi}(t) w(t) \Di t
  - 2 \int_{0}^{\infty} (\varphi(t) - \varphi_\xi(t)) \partial_{H} \nabla_{\theta} \varphi_{\xi} (t) w(t) \Di t.
\end{equation*}
Evaluated at the point
$(\varphi_{\xi_0}, \xi_0)= (\varphi_{\xi_0}, G(\varphi_{\xi_0}, H_0))$
yields the simpler expression:
\begin{equation*}
    D^{1, 2}_{\varphi_{\xi_0}, \xi_0} \nabla_{\theta} F
  = 2 \int_{0}^{\infty} \partial_H \varphi_{\xi_0}(t) \nabla_{\theta} \varphi_{\xi_0}(t) w(t) \Di t.
\end{equation*}
Suppose we are in the case $k > H + 1/\alpha$ then we may pick
$r = 2$ in the discussion above. By Fréchet
differentiability it follows that
\begin{equation}%
\label{estimdec}%
\begin{aligned}
  \sqrt{n} ( \xi_n - \xi_0) &= \sqrt{n} (G(\varphi_n, H_n) - G(\varphi_{\xi_0}, H_0))
  \\
                            &= D_{\varphi_{\xi_0}, \xi_0} G(\sqrt{n}(\varphi_n -
                              \varphi_{\xi_0}), \sqrt{n} (H_n - H_0))
  \\
                            &+ \sqrt{n} (\| \varphi_n - \varphi_{\xi_0} \|_{w, 2} + |H_n - H_0|)
                              R(\varphi_n - \varphi_{\xi_0}, H_n - H_0),
\end{aligned}
\end{equation}
where the remainder term $R$ satisfies that
$R(\varphi_n - \varphi_{\xi_0}, H_n - H_0) \asTo 0$ as
$\|\varphi_n - \varphi_{\xi_0}\|_{w, 2} + |H_n - H_0| \asTo 0$.
Recall now the derivative of $G$ at
\eqref{eq:derivative-of-implicit}. In order to show
Theorem~\ref{thm:main-result}\ref{it:thm:main-result:2} it suffices to
prove the convergences
\begin{equation}
  \label{eq:sufficient-conv}
  \begin{aligned}
    \sqrt{n} (\| \varphi_n - \varphi_{\xi_0} \|_{w, 2} + |H_n - H_0|) &\distTo \|W \|_{w, 2} + |M_1|,
    \\
   \sqrt{n} \int_{0}^{\infty} (\varphi_n(t) - \varphi_{\xi_0}(t)) \nabla_{\theta} \varphi_{\xi_0}(t) w(t) \Di t &\distTo \int_{0}^{\infty} W_t \nabla_{\theta} \varphi_{\xi_0}(t) w(t) \Di t,
\end{aligned}
\end{equation}
where $W = (W_t)_{t \geq 0}$ has been introduced in
Theorem~\ref{thm:weak-limit}\ref{it:thm:weak-limit:1}.

We will only consider the second convergence at
\eqref{eq:sufficient-conv} since the first is shown similarly (see
also \cite[page~14]{LjunPara}). For the
conditions~\eqref{eq:integral-conv-conds} it suffices to find a
constant $C > 0$ such that
\begin{equation}
  \label{eq:var-bound-W2}
  \sup_{n \in \bbN, t \geq 0} \var(W_n^2(t)) \leq C < \infty.
\end{equation}
The identity $\Delta_{i, k} X = \int_{\bbR} h_k(i - s) \Di L_s$
together with stationarity of the increments
$\{\Delta_{i, k} X \mid i \geq k \}$ shows that
\begin{equation}
\label{eq:covariance-W2}
  |\cov(W_n^2(s),W_n^2(t))| \leq \frac{1}{2} \sum_{l \in \bbZ} |U_{h_k, h_k(l + \cdot)}(s, t) + U_{h_k, -h_k(l + \cdot)}(s, t)|.
\end{equation}
Indeed, split the series at \eqref{eq:covariance-W2} into
three terms. For $l = 0$ it follows from \eqref{eq:dependence-measure}
that
\begin{equation}
  \label{eq:leq0}
  U_{h_k, h_k}(t, t) + U_{h_k, -h_k}(t, t)
  = 1 + 2 \exp(-|2t \sigma \|h_k\|_\alpha |^{\alpha}) - 2 \exp(-2|\sigma t \|h_k\|_\alpha|^{\alpha}),
\end{equation}
which is obviously uniformly bounded in $t \geq 0$. For $l \neq 0$ with
$|l| \leq k$ the first inequality of Lemma~\ref{lem:dependence-measure}
implies that
\begin{equation}
  \label{eq:lleqk}
\begin{aligned}
  \sum_{l \in \bbZ : |l| \leq k} | U_{h_k, h_k(l + \cdot)}(t, t) + U_{h_k, -h_k(l + \cdot)}(t, t) |
  &\leq 2 t^\alpha \sum_{l \in \bbZ : |l| \leq k} \rho_l \exp(-2t^\alpha(\|h_k\|_\alpha^\alpha - \rho_l))
  \\
  &\leq C t^\alpha \exp(-2 t^\alpha(\|h_k\|_\alpha^\alpha - \max_{|l| \leq k} \rho_l)).
\end{aligned}
\end{equation}
Now by Cauchy-Schwarz inequality $\rho_l < \|h_k\|_\alpha^\alpha$ for
all $l$, and we obtain a uniform bound in $t \geq 0$. By
Lemmas~\ref{lem:dependence-measure} and~\ref{lem:rho-bound} there
exist constants $C, K > 0$ such that
\begin{align}
  \MoveEqLeft\sum_{|l| > k} |U_{h_k, h_k(l + \cdot)}(t, t) + U_{h_k, -h_k(l + \cdot)}(t, t)|
  \nonumber
  \\
&\leq 2^\alpha t^\alpha \exp(-2 t^\alpha (\|h_k\|_\alpha^\alpha - \sup_{|l| > k} \rho_l)) \sum_{|l| > k} |l|^{(\alpha(H - k) - 1)/2}
  \nonumber
  \\
&\leq C t^\alpha \exp(-K t^\alpha),
  \label{eq:lgeqk}
\end{align}
where we used the assumption $k > H + 1/\alpha$ and that
$\rho_l \to 0$ for $|l| \to \infty$ by
Lemma~\ref{lem:rho-bound}. Combining \eqref{eq:leq0}, \eqref{eq:lleqk}
and~\eqref{eq:lgeqk} we can conclude \eqref{eq:var-bound-W2}. This
completes the proof of
Theorem~\ref{thm:main-result}\ref{it:thm:main-result:2}.

\subsection{Proof of
  Theorem~\ref{thm:main-result}\ref{it:thm:main-result:3}}
\label{sec:stable-case}

As in the proof of
Theorem~\ref{thm:main-result}\ref{it:thm:main-result:2} we obtain the
decomposition \eqref{estimdec}, where the convergence rate $\sqrt{n}$
is replaced by $n^{1-1/\beta}$. Furthermore, as in
\eqref{eq:sufficient-conv}, it suffices to prove that for some
$r \in (1, 2)$ then as $n \to \infty$
\begin{align*}
  n^{1-1/\beta} (\norm{\varphi_n - \varphi_{\xi_0}}_{w, r} + |H_n - H_0|)
  &\distTo \|\kappa_2 S \|_{w, r} + |M_2|,
  \\
  n^{1-1/\beta} \int_{0}^{\infty}  (\varphi_n(t) - \varphi_{\xi_0}(t)) \nabla_{\theta} \varphi_{\xi_0}(t) w(t) \Di t &\distTo S \int_{0}^{\infty} \kappa_2(t)  \nabla_{\theta} \varphi_{\xi_0}(t) w(t) \Di t.
\end{align*}
As before in the Gaussian case it is enough to provide uniform bounds
(in $n$ and $t$) on the moments in order to use
Proposition~\ref{prop:fidi-int-conv}. 

Recall that the dominating term in \eqref{eq:decomp-Sn} is given by
\begin{equation*}
  \overline{S}_n(t)
  = n^{-1/\beta} \sum_{i = k}^{n} (\overline{\Phi}_t(L_{i + 1} - L_i) - \E[\overline{\Phi}_t(L_{i + 1} - L_i)]).
\end{equation*}
Inspired by the classical case of i.i.d. random variables, each in the
domain of attraction of a stable distribution, we shall prove the
following result.

\begin{proposition}
  \label{finalprop}
For any $r \in (0, \beta)$ we have that 
\begin{equation*}
  \sup_{n \in \bbN, t \geq 0} \E[\abs{\overline{S}_n(t)}^r] < \infty.
\end{equation*}
\end{proposition}

\begin{proof}
  By Jensen's inequality it suffices to consider $r > 1$ (indeed
  $\beta \in (1,2)$). Recall the relation
  $\Phi_{t, \sigma \norm{h_k}_{\alpha}}(x) = \Phi_t^2(x) =
  \exp(-\abs{\sigma t \norm{h_k}_\alpha}^{\alpha})(\cos(t x) - 1)$
  together with
\begin{equation*}
\overline{\Phi}_t(x) = \sum_{i = 1}^{\infty} \Phi_t^2(h_k(i) x).
\end{equation*}
Note that for all $x$ in some bounded set by
Lemma~\ref{lem:bounds-Phi}\ref{it:lem:bounds-Phi:2}:
\begin{align*}
  \sup_{t \geq 0} \abs{\overline{\Phi}_t(x)}
  &\leq \sup_{t \geq 0} \exp(-\abs{\sigma \norm{h_k}_{\alpha} t}^{\alpha}) \sum_{i = 1}^{\infty} (1 \wedge (\abs{x t h_k(i)} )^2)
  \\
  &\leq \sup_{t \geq 0} \exp(-\abs{\sigma \norm{h_k}_{\alpha} t}^{\alpha}) (t + 1)^2 (\abs{x} + 1)^2 \sum_{i = 1}^{\infty} (1 \wedge \abs{h_k(i)}^2) \leq C < \infty.
\end{align*}
By \eqref{eq:asympt-Phi} there exists $x_0 < -1$ such that for all
$t > 0$
\begin{equation}
\label{eq:unif-close}
\abs[\bigg]{\frac{\abs{\overline{\Phi}_t(x)}}{\abs{x}^q} - \abs{\kappa_2(t)}} \leq 1 \qquad \text{for all $x < x_0$},
\end{equation}
where $q = \alpha/\beta$, $\beta = 1 + \alpha(k - H) \in (1, 2)$
and $\kappa_2(t) = Kt^q \exp(-\abs{t \norm{h_k}_\alpha \sigma}^\alpha)$
with $K<0$. For shorter notation we write
$D_i = L_{i + 1} - L_i$ to denote the $i$th increment of $L$. Since
$\E[\overline{\Phi}_t(D_1)]$ is bounded in $t$ we may
replace $\overline{\Phi}_t(x)$ with
$\overline{\Phi}_t(x) -
\E[\overline{\Phi}_t(D_1)]$ in \eqref{eq:unif-close} if $x_0$ is chosen large enough. 
Define
for each $t \geq 0$, $n \in \bbN$ and $i \in \{k, \ldots, n\}$
\begin{align*}
  Y_{n, t, i} &= (\overline{\Phi}_t(D_i) - \E[\overline{\Phi}_t(D_i)]) \1_{ \{\abs{ \overline{\Phi}_t(D_i) - \E[\overline{\Phi}_t(D_i)]} \leq n^{1/\beta}\} },
  \\
  Z_{n, t, i} &= (\overline{\Phi}_t(D_i) - \E[\overline{\Phi}_t(D_i)]) \1_{\{\abs{ \overline{\Phi}_t(D_i) - \E[\overline{\Phi}_t(D_i)]} > n^{1/\beta}\} }.
\end{align*}
We have the decomposition
\begin{align*}
  T_{n, t} &\coloneqq \sum_{i = k}^{n} (\overline{\Phi}_t(D_i) - \E[\overline{\Phi}_t(D_i)])
             = \sum_{i = k}^{n} (Y_{n, t, i} - \E[Y_{n, t, i}]) + \sum_{i = k}^{n} (Z_{n, t, i} - \E[Z_{n, t, i}])
  \\
           &\eqqcolon T_{n, t, 1} + T_{n, t, 2}.
\end{align*}
The proposition then asserts that
\begin{equation*}
  \sup_{n \in \bbN, t \geq 0} \E[\abs{n^{-1/\beta} T_{n, t} }^r] < \infty \qquad \text{for all $r \in (1, \beta)$.}
\end{equation*}
To prove this we observe that 
\begin{equation*}
  \E[\abs{n^{-1/\beta} T_{n, t}}^r] \leq C_r \bigl( \E[\abs{n^{-1/\beta} T_{n, t, 1}}^r] +  \E[\abs{n^{-1/\beta} T_{n, t, 2}}^r] \bigr).
\end{equation*}
For the first term we obtain the inequality
\begin{equation*}
  \E[\abs{n^{-1/\beta} T_{n, t, 1}}^r]
  \leq \E\bigl[ \abs{ n^{-1/\beta} T_{n, t, 1} }^2 \bigr]^{r/2}
  \leq C_r (n^{1 - 2/\beta} \E[ \abs{Y_{n, t, k}}^2 ])^{r/2}.
\end{equation*}
For short notation let $E_t = \E[\overline{\Phi}_t(D_1)]$, which is
uniformly bounded in $t \geq 0$. Additionally let $p_{\alpha}$ denote
the density of an S$\alpha$S distribution and recall that
$p_\alpha(x) \leq C (1 + \abs{x})^{-1 - \alpha}$ for all $x \in \bbR$,
cf. \cite[Theorem~1.1]{WataAsym}. We decompose
$\E[\abs{Y_{n, t, k}}^2]$ into two regions corresponding to
\eqref{eq:unif-close}:
\begin{align*}
  n^{1 - 2/\beta} \E[\abs{Y_{n, t, k}}^2]
  &= 2 n^{1 - 2/\beta} \int^{0}_{x_0} \abs{\overline{\Phi}_t(x) - E_t}^2 \1_{ \{\abs{ \overline{\Phi}_t(x) - E_t} \leq n^{1/\beta}\} } p_{\alpha}(x) \Di x
  \\
  &+ 2 n^{1 - 2/\beta} \int^{x_0}_{-\infty} \abs{\overline{\Phi}_t(x) - E_t}^2 \1_{ \{\abs{ \overline{\Phi}_t(x) - E_t} \leq n^{1/\beta\}} } p_{\alpha}(x)  \Di x.
\end{align*}
The first term vanishes as $n \to \infty$ since $2/\beta > 1$ and the
fact that $\abs{\overline{\Phi}_t(x) - E_t}$ is bounded 
for all $t$ and $x \in (x_0, 0)$. The second term is further  split into two terms:
\begin{align*}
\MoveEqLeft  n^{1 - 2/\beta} \int_{-\infty}^{x_0} \abs{ \overline{\Phi}_t(x) - E_t}^2 \1_{\{\abs{\overline{\Phi}_t(x) - E_t} \leq n^{1/\beta}\}} p_{\alpha}(x) \Di x
  \\
&= n^{1 - 2/\beta} \int_{-\infty}^{x_0} \abs{ \overline{\Phi}_t(x) - E_t}^2 \1_{\{\abs{x}^q < \abs{\overline{\Phi}_t(x) - E_t} \leq n^{1/\beta}\}} p_{\alpha}(x) \Di x
\\
&+ n^{1 - 2/\beta} \int_{-\infty}^{x_0} \abs{ \overline{\Phi}_t(x) - E_t}^2 \1_{\{\abs{\overline{\Phi}_t(x) - E_t} \leq n^{1/\beta} \wedge \abs{x}^q\}} p_{\alpha}(x) \Di x.
\end{align*}
Using \eqref{eq:unif-close} and the boundedness of $\kappa_2$ on the
first term we have that
\begin{align*}
  \MoveEqLeft n^{1 - 2/\beta} \int_{-\infty}^{x_0} \abs{ \overline{\Phi}_t(x) - E_t}^2 \1_{\{\abs{x}^q < \abs{\overline{\Phi}_t(x) - E_t} \leq n^{1/\beta}\}} p_{\alpha}(x) \Di x
  \\
&\leq n^{1 - 2/\beta} \int_{-x_0}^{\infty} (\abs{\kappa_2(t)} + 1)^2 x^{2 q} \1_{\{x^q \leq n^{1/\beta}\}} p_\alpha(x) \Di x
  \\
&\leq C_q n^{1 - 2/\beta} \int_{-x_0}^{n^{1/q\beta}} x^{2q - 1 - \alpha} \Di x
  = C_q n^{1- 2/\beta}(1 + n^{(2q - \alpha)/q\beta}) \leq C_q.
\end{align*}
The second term contains a similar consideration, indeed
\begin{align*}
  \MoveEqLeft  n^{1 - 2/\beta} \int_{-\infty}^{x_0} \abs{ \overline{\Phi}_t(x) - E_t}^2 \1_{\{\abs{\overline{\Phi}_t(x) - E_t} \leq n^{1/\beta} \wedge \abs{x}^q\}} p_{\alpha}(x) \Di x
  \leq \int_{-x_0}^{\infty} (n^{1/\beta} \wedge x^q)^2 p_\alpha(x) \Di x
  \\
&= n^{1 - 2/\beta} \int_{-x_0}^{n^{1/q\beta}} x^{2q} p_\alpha(x) \Di x
  + n^{1 - 2/\beta} \int_{n^{1/q\beta}}^{\infty} n^{2/\beta} p_\alpha(x) \Di x \leq C_q.
\end{align*}
In the next step we treat the term $T_{n, t, 2}$. Note first that
by the von Bahr-Esseen inequality we obtain
\begin{equation*}
  \E[\abs{n^{-1/\beta} T_{n, t, 2}}^r] \leq C_r n^{1 - r/\beta} \E[\abs{Z_{n, t, k}}^r].
\end{equation*}
Decomposing as above we have that
\begin{align*}
  n^{1 - r/\beta} \E[\abs{Z_{n, t, k}}^r]
  &= 2 n^{1 - r/\beta} \int^{0}_{x_0} \abs{\overline{\Phi}_t(x) - E_t}^{r} \1_{ \{\abs{ \overline{\Phi}_t(x) - E_t} > n^{1/\beta}\} } p_{\alpha}(x) \Di x
  \\
  &+ 2 n^{1 -r/\beta} \int_{-\infty}^{x_0} \abs{\overline{\Phi}_t(x) - E_t}^{r} \1_{ \{\abs{ \overline{\Phi}_t(x) - E_t} > n^{1/\beta}\} } p_{\alpha}(x)  \Di x.
\end{align*}
For the first term we recall that $\overline{\Phi}_t(x)$ is bounded
uniformly in $t$ when $x$ lies in a bounded set, hence
$n^{1/\beta} > \abs{\overline{\Phi}_t(x) - E_t}$ for all sufficiently
large $n$, independent of $x \in (0,x_0)$ and $t \geq 0$, so the first
term is zero for sufficiently large $n$.

The last term requires more computations. Due to \eqref{eq:unif-close}
and the fact that $\kappa_2$ is bounded it follows that
\begin{align*}
  \MoveEqLeft n^{1 - r/\beta} \int_{-\infty}^{x_0} \abs{\overline{\Phi}_t(x) - E_t}^{r} \1_{ \{\abs{ \overline{\Phi}_t(x) - E_t} > n^{1/\beta}\} } p_{\alpha}(x)  \Di x
  \\
&\leq C n^{1 - r/\beta} \int_{-\infty}^{x_0} (\abs{\kappa_2(t)} + 1)^r \abs{x}^{r q} \1_{ \{(\abs{\kappa_2(t)} + 1) \abs{x}^q > n^{1/\beta}\} } p_{\alpha}(x) \Di x
  \\
&\leq C n^{1 - r/\beta} \int_{-\infty}^{x_0} \abs{x}^{r q - 1 - \alpha} \1_{\{\abs{x} > n^{1/q \beta}/K \}} \Di x
  \\
&= C n^{1 - r/\beta} \int^{\infty}_{n^{1/\alpha}/K } x^{r q - 1 - \alpha} \Di x
  \leq C,
\end{align*}
where we used that $rq - \alpha < 0$ since $r < \beta$.
\end{proof}

\noindent
Combining Propositions~\ref{prop:R-moment}--\ref{finalprop} we finally
complete the proof of
Theorem~\ref{thm:main-result}\ref{it:thm:main-result:3}.

\clearpage

\section{Tables}
\label{sec:tables}

\begin{table}[htpb!]
  \begin{minipage}{0.45\linewidth}
    \caption{Absolute value of the bias based on $n = \num{1000}$,
      $k = 2$, $p = -0.4$, $\nu = 0.1$ and $\sigma_0 = 0.3$ for the
      minimal contrast estimator}
\label{tab:abs-bias-est-1000}
\centering%
\begin{tabular}{l l *{3}{S[group-separator={}]}}
  \toprule
  $H_0$ & $\alpha_0$ & \multicolumn{1}{c}{$\sigma_n$} & \multicolumn{1}{c}{$\alpha_n$} & \multicolumn{1}{c}{$H_n(p, k)$}
  \\
  \midrule
  0.2 & \textbf{0.4} & 6.11528766806447 & 0.40780935809849 & 0.201534114406659
  \\ 
        & 0.6 & 0.0997920838339353 & 0.00758888133430636 & 0.181844674118716
  \\ 
        & 0.8 & 0.0416576040294591 & 0.0018170528212031 & 0.134583706736434
  \\ 
        & 1 & 0.0539747747690746 & 1.95088288516879e-05 & 0.113670978295413
  \\ 
        & 1.2 & 0.0451027056347307 & 0.000768928702441876 & 0.0881080203851828
  \\ 
        & 1.4 & 0.0367100883073913 & 0.00446776349498455 & 0.0702200043064256
  \\ 
        & 1.6 & 0.0308579424298173 & 0.020686173093613 & 0.0705931712654161
  \\ 
        & 1.8 & 0.011747693750663 & 0.0535148407681887 & 0.0643482859989043
  \\ 
  \midrule
  0.4  & \textbf{0.4} & 5.4235985606574 & 0.328143625225551 & 0.10402554735686
  \\ 
        & \textbf{0.6} & 0.0773963498130176 & 0.00350651747936461 & 0.123674360784111
  \\ 
        & 0.8 & 0.0472012481879438 & 0.005120351956535 & 0.0747580079950212
  \\ 
        & 1 & 0.0322343116850873 & 0.000312616999742026 & 0.0500901597618562
  \\ 
        & 1.2 & 0.0275704144171758 & 0.000952014549291298 & 0.0351239934509173
  \\ 
        & 1.4 & 0.0114503702578297 & 0.00904576172483804 & 0.0188717550722743
  \\ 
        & 1.6 & 0.00718847774422977 & 0.0486757169944044 & 0.0261154496906544
  \\ 
        & 1.8 & 0.0496134192121276 & 0.122866212975795 & 0.0192878904957188
  \\ 
  \midrule
  0.6  & \textbf{0.4} & 3.61725666013326 & 0.200748676561363 & 0.0637540701337884
  \\ 
        & \textbf{0.6} & 0.0492867592316138 & 0.00583950463445159 & 0.0608112173117509
  \\ 
        & 0.8 & 0.0375493868444194 & 0.00358585899204208 & 0.0401964970469688
  \\ 
        & 1 & 0.0328143052391076 & 0.00198956792610263 & 0.023475714900871
  \\ 
        & 1.2 & 0.0115397222084812 & 0.0161369017308843 & 0.0186835893581606
  \\ 
        & 1.4 & 0.00668341346782603 & 0.0327298030064071 & 0.00652475618396034
  \\ 
        & 1.6 & 0.0366272511617723 & 0.0933433632587354 & 0.0111688320698823
  \\ 
        & 1.8 & 0.112945464591753 & 0.250109587269972 & 0.00893913432966442
  \\ 
  \midrule
  0.8  & \textbf{0.4} & 1.69064774181509 & 0.086617951280715 & 0.0203933408544415
  \\ 
        & \textbf{0.6} & 0.0483287534287389 & 0.0027305095109453 & 0.0403559465809294
  \\ 
        & \textbf{0.8} & 0.0513635928683994 & 0.00252464106972363 & 0.0302594769239117
  \\ 
        & 1 & 0.0345338608312343 & 0.00272201453336093 & 0.0167477861547928
  \\ 
        & 1.2 & 0.00531187592719367 & 0.0214864421598182 & 0.0106950113151211
  \\ 
        & 1.4 & 0.0160127341102578 & 0.0510512184462523 & 0.00565770101145122
  \\ 
        & 1.6 & 0.0701925307642103 & 0.142488294105705 & 0.00137226627006914
  \\ 
        & 1.8 & 0.172402661356385 & 0.391454035999514 & 0.0016089240533326
  \\ 
  \bottomrule
\end{tabular}
\end{minipage}%
\hfill%
\begin{minipage}{0.45\linewidth}
  \caption{Standard deviation based on $n = \num{1000}$, $k = 2$,
    $p = -0.4$, $\nu = 0.1$ and $\sigma_0 = 0.3$ for the
      minimal contrast estimator}
  \centering%
\label{tab:std-est-1000}
\begin{tabular}{l l *{3}{S[group-separator={}]}}
  \toprule
  $H_0$ & $\alpha_0$ & \multicolumn{1}{c}{$\sigma_n$} & \multicolumn{1}{c}{$\alpha_n$} & \multicolumn{1}{c}{$H_n(p, k)$}
  \\
  \midrule
  0.2  & \textbf{0.4} & 7.41042581291778 & 0.458541106519428 & 0.160884198224583
  \\ 
        & 0.6 & 0.428934926607785 & 0.0925333774972922 & 0.162209611752009
  \\ 
        & 0.8 & 0.244507102913203 & 0.0680756434262012 & 0.146843029332103
  \\ 
        & 1 & 0.204718694184439 & 0.0773307805746779 & 0.135523718270079
  \\ 
        & 1.2 & 0.183134585482617 & 0.0922811912530568 & 0.128295596084238
  \\ 
        & 1.4 & 0.173338410536434 & 0.111831033533225 & 0.124415483966601
  \\ 
        & 1.6 & 0.155148840922753 & 0.137649906210743 & 0.121636048052211
  \\ 
        & 1.8 & 0.141537127518805 & 0.146112494398493 & 0.120958267540089
  \\ 
  \midrule
  0.4  & \textbf{0.4} & 7.93134366583018 & 0.440657826266336 & 0.166351346459474
  \\ 
        & \textbf{0.6} & 0.361923302209863 & 0.0717212743312573 & 0.172282948595995
  \\ 
        & 0.8 & 0.229462187364255 & 0.0685701827560679 & 0.154645655537333
  \\ 
        & 1 & 0.186419240839882 & 0.0868704197066231 & 0.150323313183463
  \\ 
        & 1.2 & 0.16983688621395 & 0.115130365694339 & 0.144812405431106
  \\ 
        & 1.4 & 0.158955772743318 & 0.15834314691447 & 0.138390356017191
  \\ 
        & 1.6 & 0.152656202558975 & 0.196350214371991 & 0.144687703357566
  \\ 
        & 1.8 & 0.125485736691118 & 0.195873740196142 & 0.139719517396656
  \\ 
  \midrule
  0.6  & \textbf{0.4} & 7.35887288651317 & 0.34971687147606 & 0.192302067656162
  \\ 
        & \textbf{0.6} & 0.256143369417277 & 0.065190106057068 & 0.17289012832241
  \\ 
        & 0.8 & 0.194787549118206 & 0.0675876092806175 & 0.157946098741844
  \\ 
        & 1 & 0.181675194243656 & 0.102979949796495 & 0.153820659188558
  \\ 
        & 1.2 & 0.168310489505967 & 0.141098951159374 & 0.140589194727408
  \\ 
        & 1.4 & 0.165118148853579 & 0.197420527887819 & 0.149754582428995
  \\ 
        & 1.6 & 0.15594573786182 & 0.243111247705317 & 0.146508760490694
  \\ 
        & 1.8 & 0.117131627017047 & 0.227088240601706 & 0.140966084696083
  \\ 
  \midrule
  0.8  & \textbf{0.4} & 5.07789744674506 & 0.218332280411531 & 0.184848686844498
  \\ 
        & \textbf{0.6} & 0.253899072928444 & 0.0632487308235487 & 0.175825911507481
  \\ 
        & \textbf{0.8} & 0.20137084063814 & 0.0747552431698255 & 0.152321793871011
  \\ 
        & 1 & 0.186424012898282 & 0.113992241102561 & 0.1545778713509
  \\ 
        & 1.2 & 0.184337663418976 & 0.176534505928454 & 0.15334933371098
  \\ 
        & 1.4 & 0.174935935224409 & 0.235345358830952 & 0.142363213241829
  \\ 
        & 1.6 & 0.140245811603672 & 0.241335044774823 & 0.141669828892496
  \\ 
        & 1.8 & 0.115602947077765 & 0.234208127513804 & 0.136063606637985
  \\ 
  \bottomrule
\end{tabular}
\end{minipage}
\end{table}

\begin{table}
  \begin{minipage}{0.45\linewidth}
    \caption{Absolute value of the bias based on $n = \num{10000}$,
      $k = 2$, $p = -0.4$, $\nu = 0.1$ and $\sigma_0 = 0.3$ for the
      minimal contrast estimator}
    \label{tab:abs-bias-est}
    \centering%
    \begin{tabular}{l l *{3}{S[group-separator={}]}}
      \toprule%
      $H_0$ & $\alpha_0$ & \multicolumn{1}{c}{$\sigma_n$} & \multicolumn{1}{c}{$\alpha_n$} & \multicolumn{1}{c}{$H_n(p, k)$}
      \\
      \midrule
      0.2  & \textbf{0.4} & 2.27992311452631 & 0.141477363340558 & 0.256308555584469
      \\ 
            & 0.6 & 0.0105572793676621 & 0.00183317141348813 & 0.165167959596669
      \\ 
            & 0.8 & 0.00317130441757359 & 0.00248554732554923 & 0.113388301533688
      \\ 
            & 1 & 0.00461126334365977 & 0.00052279671821219 & 0.0847511035925393
      \\ 
            & 1.2 & 0.00780929665020615 & 0.00115628690355876 & 0.0635890368642597
      \\ 
            & 1.4 & 0.0110153061180726 & 0.00149755387138249 & 0.0516469324189002
      \\ 
            & 1.6 & 0.0107281585329711 & 0.00122686795671139 & 0.0412873292769312
      \\ 
            & 1.8 & 0.007671980667453 & 0.000960570071134368 & 0.0317861856208238
      \\ 
      \midrule
      0.4  & \textbf{0.4} & 0.662903752876743 & 0.0457881099183104 & 0.176593442098076
      \\ 
            & \textbf{0.6} & 0.00806127483429993 & 0.000295442282905916 & 0.102215388677313
      \\ 
            & 0.8 & 0.00500404957224098 & 0.00172704903753494 & 0.0600176172953674
      \\ 
            & 1 & 0.00178056176822937 & 0.000989156823092539 & 0.0365556998380076
      \\ 
            & 1.2 & 0.00519293019094857 & 0.00180450367176329 & 0.0233776670663156
      \\ 
            & 1.4 & 0.00678114922642133 & 0.00189983829137586 & 0.0169551078016329
      \\ 
            & 1.6 & 0.00974261380766353 & 0.0065433384748777 & 0.0103432936067711
      \\ 
            & 1.8 & 0.0123981278768924 & 0.0145058863658259 & 0.00408513019335098
      \\ 
      \midrule
      0.6  & \textbf{0.4} & 0.151721156306062 & 0.0171174208048925 & 0.112956923359921
      \\ 
            & \textbf{0.6} & 0.00212393345808329 & 0.00101432896754304 & 0.0588116779169856
      \\ 
            & 0.8 & 0.00670397392200299 & 0.00145853457712767 & 0.0313711643729745
      \\ 
            & 1 & 0.00567731649170062 & 0.00109199660291098 & 0.0155763867382522
      \\ 
            & 1.2 & 0.00254937865865262 & 0.00078291413296926 & 0.00698787774263924
      \\ 
            & 1.4 & 0.00317992876454468 & 0.000975428453787155 & 0.0034149867214798
      \\ 
            & 1.6 & 0.00812426410328695 & 0.00499159238105709 & 0.000780802174706853
      \\ 
            & 1.8 & 0.0142895955241071 & 0.0125853027782491 & 0.00170889568254926
      \\ 
      \midrule
      0.8  & \textbf{0.4} & 0.0574826907292774 & 0.00596824344193844 & 0.079291016352223
      \\ 
            & \textbf{0.6} & 0.005283617044454 & 0.00129175010761531 & 0.0335310924589099
      \\ 
            & \textbf{0.8} & 0.00333206176419313 & 0.000185366666490064 & 0.013050927405611
      \\ 
            & 1 & 0.00488108154474771 & 0.000373064649762268 & 0.0065139690602887
      \\ 
            & 1.2 & 2.60005989485496e-05 & 0.000583475507690175 & 0.00142778292335807
      \\ 
            & 1.4 & 0.00301624116535638 & 0.00139996525751885 & 0.000424273152560514
      \\ 
            & 1.6 & 0.0094828037569356 & 0.00445940090248733 & 0.00160177689724888
      \\ 
            & 1.8 & 0.00387916558309527 & 0.00354865497469016 & 0.000159946285588918
      \\ 
      \bottomrule
\end{tabular}
\end{minipage}%
\hfill%
\begin{minipage}{0.45\linewidth}
  \caption{Standard deviation based on $n = \num{10000}$, $k = 2$,
    $p = -0.4$, $\nu = 0.1$ and $\sigma_0 = 0.3$ for the
      minimal contrast estimator}
  \centering%
\label{tab:std-est}
\begin{tabular}{l l *{3}{S[group-separator={}]}}
  \toprule
  $H_0$ & $\alpha_0$ & \multicolumn{1}{c}{$\sigma_n$} & \multicolumn{1}{c}{$\alpha_n$} & \multicolumn{1}{c}{$H_n(p, k)$}
  \\
  \midrule
  0.2 & \textbf{0.4} & 4.46797947671098 & 0.275947548266881 & 0.0693320825070079
  \\ 
        & 0.6 & 0.0983358032785961 & 0.0291980560647043 & 0.0559629374113014
  \\ 
        & 0.8 & 0.106163290012163 & 0.031320459313742 & 0.0509529732068706
  \\ 
        & 1 & 0.064223544552912 & 0.0236168575429189 & 0.0498575722360116
  \\ 
        & 1.2 & 0.06841379216024 & 0.0354330877365283 & 0.0484144031589578
  \\ 
        & 1.4 & 0.0699040094740681 & 0.0506095522787539 & 0.0489079093603632
  \\ 
        & 1.6 & 0.0711233438988914 & 0.0665727649509567 & 0.0478850607168149
  \\ 
        & 1.8 & 0.0653885121831966 & 0.0678379662793237 & 0.0475714410663651
  \\ 
  \midrule
  0.4  & \textbf{0.4} & 2.23704945227845 & 0.145846014863326 & 0.0704061444268706
  \\ 
        & \textbf{0.6} & 0.082583235940031 & 0.0230401353761834 & 0.0628203182831885
  \\ 
        & 0.8 & 0.0933285570692582 & 0.0305367129337971 & 0.0508814709042337
  \\ 
        & 1 & 0.0576835912301131 & 0.0309469542041297 & 0.0504273798534364
  \\ 
        & 1.2 & 0.0574360812636387 & 0.0393416409282213 & 0.0493747514718563
  \\ 
        & 1.4 & 0.0647448363188212 & 0.0619849049891934 & 0.0465175578471279
  \\ 
        & 1.6 & 0.081236573124305 & 0.0958767116513674 & 0.0483952940107769
  \\ 
        & 1.8 & 0.0824867427964494 & 0.116808098809584 & 0.0479100228621468
  \\ 
  \midrule
  0.6  & \textbf{0.4} & 0.695366733566539 & 0.0689471076134367 & 0.0713122955094812
  \\ 
        & \textbf{0.6} & 0.101373344245252 & 0.0225925446197606 & 0.0576209966275316
  \\ 
        & 0.8 & 0.0755640900055119 & 0.0289988575832692 & 0.050610715577751
  \\ 
        & 1 & 0.0589368264797922 & 0.0327368039799709 & 0.0480606835686756
  \\ 
        & 1.2 & 0.0545605075596676 & 0.0445599063563097 & 0.0480099126523805
  \\ 
        & 1.4 & 0.0710158344236858 & 0.0748553940914841 & 0.0479244728781267
  \\ 
        & 1.6 & 0.0955363340999473 & 0.123813056530805 & 0.0461690775291001
  \\ 
        & 1.8 & 0.0956950661837132 & 0.146903933061165 & 0.0469013703058911
  \\ 
  \midrule
  0.8 & \textbf{0.4} & 0.35706860656807 & 0.050955972228937 & 0.0748942564194485
  \\ 
        & \textbf{0.6} & 0.129032363877391 & 0.0256380001422129 & 0.0596259711432178
  \\ 
        & \textbf{0.8} & 0.0524018510469036 & 0.0244656219392922 & 0.0514832475564527
  \\ 
        & 1 & 0.0645263710298586 & 0.0388483533246653 & 0.0496843781920851
  \\ 
        & 1.2 & 0.0669693866641396 & 0.0592620758007229 & 0.0474774714961547
  \\ 
        & 1.4 & 0.0845184424365827 & 0.0950586615371772 & 0.0469918991675485
  \\ 
        & 1.6 & 0.111875247870183 & 0.151207466903356 & 0.0481624620634028
  \\ 
        & 1.8 & 0.0987359819888986 & 0.162050484618398 & 0.0508444751267591
  \\ 
  \bottomrule
  \end{tabular}
\end{minipage}
\end{table}

\begin{table}[htpb!]
  \begin{minipage}{0.45\linewidth}
    \caption{Absolute value of bias based on $n = \num{10000}$,
      $k = 1$, $p = -0.4$, $\nu = 0.1$ and $\sigma_0 = 0.3$}
    \label{tab:abs-bias-est-1}
    \begin{tabular}{l l *{3}{S[group-separator={}]}}
      \toprule
      $H_0$ & $\alpha_0$ & \multicolumn{1}{c}{$\sigma_n$} & \multicolumn{1}{c}{$\alpha_n$} & \multicolumn{1}{c}{$H_n(p, k)$}
      \\
      \midrule
      0.2  & \textbf{0.4} & 0.020159582477903 & 0.0176570880403363 & 0.187253280294767
      \\ 
            & \textbf{0.6} & 0.0671346849184927 & 0.0001448626648684 & 0.132858880135923
      \\ 
            & \textbf{0.8} & 0.0256298998370676 & 0.00148057499652728 & 0.0966038278908169
      \\ 
            & \textbf{1} & 0.00436145371956465 & 0.00281406841474508 & 0.0745335743715258
      \\ 
            & \textbf{1.2} & 0.00256130877306552 & 0.00177956499346869 & 0.0589043056921438
      \\ 
            & 1.4 & 0.0096720772946569 & 0.00653126597799948 & 0.0465190791043708
      \\ 
            & 1.6 & 0.01608757625496 & 0.0134982875370081 & 0.0375678690938378
      \\ 
            & 1.8 & 0.0215894718172951 & 0.0269865302518468 & 0.0285928493137589
      \\ 
      \midrule
      0.4  & \textbf{0.4} & 0.0655631102365738 & 0.00896219764010321 & 0.0823890306285371
      \\ 
            & \textbf{0.6} & 0.0623201428993455 & 0.00166327224374958 & 0.0563757169821656
      \\ 
            & \textbf{0.8} & 0.0201151404801131 & 0.00201927820024862 & 0.0393528006598102
      \\ 
            & \textbf{1} & 0.00165823578739779 & 0.00370795039678708 & 0.0277768836667843
      \\ 
            & \textbf{1.2} & 0.00759232951034614 & 0.00521893519988822 & 0.0201639444819741
      \\ 
            & \textbf{1.4} & 0.0108918359495488 & 0.0088724372687711 & 0.012832476633058
      \\ 
            & \textbf{1.6} & 0.0229389048531876 & 0.0270602128008938 & 0.0066033223099742
      \\ 
            & 1.8 & 0.0177485444820319 & 0.0288428378930323 & 0.00484121020058942
      \\ 
      \midrule
      0.6  & \textbf{0.4} & 0.0658795671597247 & 0.0109547853093714 & 0.0108922690016445
      \\ 
            & \textbf{0.6} & 0.0723948465696135 & 0.0018343839087213 & 0.0114431506547425
      \\ 
            & \textbf{0.8} & 0.0253962159112109 & 0.00216319234692184 & 0.00246032404862595
      \\ 
            & \textbf{1} & 0.00499022788420563 & 0.0100998294678863 & 0.000933549215541333
      \\ 
            & \textbf{1.2} & 0.0131316935030757 & 0.0126862018808553 & 0.00018371423047418
      \\ 
            & \textbf{1.4} & 0.0188901893883948 & 0.0193540381488496 & 0.00270530038529816
      \\ 
            & \textbf{1.6} & 0.00826733223178444 & 0.0166739373915706 & 0.000427747214895326
      \\ 
            & \textbf{1.8} & 0.0241377537987666 & 0.0456223364196561 & 0.00138660912769006
      \\ 
      \midrule
      0.8  & \textbf{0.4} & 0.0812802652577354 & 0.0115908006162218 & 0.118350108356732
      \\ 
            & \textbf{0.6} & 0.0936543851014804 & 0.0038912244733965 & 0.0825446122227345
      \\ 
            & \textbf{0.8} & 0.0343047745421472 & 0.0117303933858765 & 0.0498897580510652
      \\ 
            & \textbf{1} & 0.00328150394616429 & 0.0194138722459878 & 0.0258017783497017
      \\ 
            & \textbf{1.2} & 0.00351005208081104 & 0.0120071465407031 & 0.000489806100466626
      \\ 
            & \textbf{1.4} & 0.00117348526915177 & 0.0196916371855729 & 0.00337185448877478
      \\ 
            & \textbf{1.6} & 0.0161876697205415 & 0.0438620319326012 & 0.00544310736208477
      \\ 
            & \textbf{1.8} & 0.0213795273901419 & 0.0570517299392745 & 0.00344456071322175
      \\ 
      \bottomrule
    \end{tabular}
  \end{minipage}\hfill
  \begin{minipage}{0.45\linewidth}
    \caption{Standard deviation based on $n = \num{10000}$, $k = 1$,
      $p = -0.4$, $\nu = 0.1$ and $\sigma_0 = 0.3$}
    \label{tab:std-est-1}
    \begin{tabular}{l l *{3}{S[group-separator={}]}}
      \toprule
      $H_0$ & $\alpha_0$ & \multicolumn{1}{c}{$\sigma_n$} & \multicolumn{1}{c}{$\alpha_n$} & \multicolumn{1}{c}{$H_n(p, k)$}
      \\
      \midrule
      0.2  & \textbf{0.4} & 0.653775926984573 & 0.0803104910567001 & 0.0655384837968519
      \\ 
            & \textbf{0.6} & 0.0681277727513606 & 0.0246536314906187 & 0.0548321245922
      \\ 
            & \textbf{0.8} & 0.0598855372372643 & 0.0239131340141106 & 0.0500948074026087
      \\ 
            & \textbf{1} & 0.069342093244775 & 0.0324140852293537 & 0.047351365394555
      \\ 
            & \textbf{1.2} & 0.0633379501629091 & 0.0422173451256746 & 0.045881040994034
      \\ 
            & 1.4 & 0.0710613483066818 & 0.0598126174067686 & 0.0461283878433596
      \\ 
            & 1.6 & 0.0835445724536796 & 0.0901470849414464 & 0.0457575218504134
      \\ 
            & 1.8 & 0.0840508364618504 & 0.107132667669429 & 0.0437093894253527
      \\ 
      \midrule
      0.4  & \textbf{0.4} & 0.354716861662041 & 0.0625262966020462 & 0.0637753873123791
      \\ 
            & \textbf{0.6} & 0.0768657642797561 & 0.0270656189382835 & 0.0561907111298767
      \\ 
            & \textbf{0.8} & 0.0754050009351268 & 0.0354601592725627 & 0.049351017473884
      \\ 
            & \textbf{1} & 0.0549397642628535 & 0.0408602818513291 & 0.0472775310204837
      \\ 
            & \textbf{1.2} & 0.0636205300651236 & 0.0560306863210051 & 0.0478422627095918
      \\ 
            & \textbf{1.4} & 0.0750702805208725 & 0.0777360608738432 & 0.0451746985688501
      \\ 
            & \textbf{1.6} & 0.100287273105692 & 0.124204675827076 & 0.0458969945114568
      \\ 
            & 1.8 & 0.0685064137135058 & 0.109550814294112 & 0.0440815336381742
      \\ 
      \midrule
      0.6 & \textbf{0.4} & 0.800491828268879 & 0.0619465215938737 & 0.0636553073999221
      \\ 
            & \textbf{0.6} & 0.0667743991168754 & 0.034160643583785 & 0.05437193981272
      \\ 
            & \textbf{0.8} & 0.0722449333224281 & 0.0512287996106751 & 0.0521374206430518
      \\ 
            & \textbf{1} & 0.0695254014094585 & 0.0574537374774565 & 0.0494234639506415
      \\ 
            & \textbf{1.2} & 0.077244885126627 & 0.0716826916201574 & 0.0453388015650444
      \\ 
            & \textbf{1.4} & 0.106467306103928 & 0.119366913192232 & 0.0458464346850768
      \\ 
            & \textbf{1.6} & 0.0776639191067291 & 0.113778645998865 & 0.0440559470371416
      \\ 
            & \textbf{1.8} & 0.0647963589236112 & 0.123782587423012 & 0.0449250667407215
      \\ 
      \midrule
      0.8  & \textbf{0.4} & 0.739830787257985 & 0.0648432244252129 & 0.0659598005341707
      \\ 
            & \textbf{0.6} & 0.104227137707133 & 0.0448364681072762 & 0.0550410992546863
      \\ 
            & \textbf{0.8} & 0.11822951559463 & 0.072810664897135 & 0.0505830808694971
      \\ 
            & \textbf{1} & 0.0845131030028341 & 0.0829641827677943 & 0.0459593221715464
      \\ 
            & \textbf{1.2} & 0.0763548684280778 & 0.0746539852826236 & 0.0443743446875042
      \\ 
            & \textbf{1.4} & 0.0629890760333985 & 0.104248826334257 & 0.0464077638633903
      \\ 
            & \textbf{1.6} & 0.0699625183540019 & 0.138950875156497 & 0.0451905497148409
      \\ 
            & \textbf{1.8} & 0.0601597768227649 & 0.141220427335772 & 0.0419122584020285
      \\ 
      \bottomrule
    \end{tabular}
  \end{minipage}
\end{table}

\begin{table}
  \begin{minipage}[b]{0.45\linewidth}
    \caption{Absolute value of bias for
      $(\widetilde{\sigma}_{\textrm{low}},
      \widetilde{\alpha}_{\textrm{low}})$ for $n = \num{10000}$,
      $p = -0.4$ and $\sigma_0 = 0.3$}
    \label{tab:abs-bias-est-old}
\begin{tabular}{l l *{2}{S[group-separator={}]}}
  \toprule
  $H_0$ & $\alpha_0$ & \multicolumn{1}{c}{$\widetilde{\sigma}_{\textrm{low}}$} & \multicolumn{1}{c}{$\widetilde{\alpha}_{\textrm{low}}$}
  \\
  \midrule
  0.2 & 0.4 & 0.287217572980045 & 0.0371086575027604
  \\ 
        & 0.6 & 0.28602154323148 & 0.173655850248536
  \\ 
        & 0.8 & 0.272458137984446 & 0.381736739034421
  \\ 
        & 1 & 0.239024112849802 & 0.488926440400127
  \\ 
        & 1.2 & 0.0887863529704307 & 0.181839225689885
  \\ 
        & 1.4 & 0.00645961473487409 & 0.00442894355110161
  \\ 
        & 1.6 & 0.00769664776985585 & 0.00802027904385211
  \\ 
        & 1.8 & 0.00471072436754421 & 0.00134944316173926
  \\ 
  \midrule
  0.4 & 0.4 & 0.280459436488355 & 0.036375700397655
  \\ 
        & 0.6 & 0.269284076982186 & 0.160561673103358
  \\ 
        & 0.8 & 0.205309911959336 & 0.293809841059945
  \\ 
        & 1 & 0.0543295733417487 & 0.137449563844791
  \\ 
        & 1.2 & 0.00524579626506824 & 0.00425926440252045
  \\ 
        & 1.4 & 0.00120516557545762 & 0.0010721468718431
  \\ 
        & 1.6 & 0.000481734470633227 & 0.000969352787131544
  \\ 
        & 1.8 & 0.000585327869237347 & 0.000701262135562804
  \\
  \midrule
  0.6 & 0.4 & 0.276431562230691 & 0.0459988795571353
  \\ 
        & 0.6 & 0.228459184667614 & 0.13617557380459
  \\ 
        & 0.8 & 0.175164278393567 & 0.258269810814984
  \\ 
        & 1 & 0.00772639891209914 & 0.0135337995028149
  \\ 
        & 1.2 & 0.00196789091939961 & 0.00227902511157954
  \\ 
        & 1.4 & 0.000134221025545529 & 0.000107850657605655
  \\ 
        & 1.6 & 0.000243251031788892 & 0.000514788419206582
  \\ 
        & 1.8 & 0.000636808368975185 & 1.64377920099549e-05
  \\ 
  \midrule
  0.8  & 0.4 & 0.26879689263687 & 0.0713213318294936
  \\ 
        & 0.6 & 0.138858777991656 & 0.0917349986960825
  \\ 
        & 0.8 & 0.00962234016457732 & 0.028599822115716
  \\ 
        & 1 & 0.00615043644160236 & 0.00734675368202823
  \\ 
        & 1.2 & 0.00127996385176149 & 4.36103423703114e-05
  \\ 
        & 1.4 & 0.000275426312166121 & 0.00232368926213752
  \\ 
        & 1.6 & 0.000100523237334137 & 0.000981428369866796
  \\ 
        & 1.8 & 4.5311136075401e-05 & 0.00199717797742566
  \\ 
  \bottomrule
    \end{tabular}
  \end{minipage}%
  \hfill%
  \begin{minipage}[b]{0.45\linewidth}
    \caption{Standard deviation of
      $(\widetilde{\sigma}_{\textrm{low}},
      \widetilde{\alpha}_{\textrm{low}})$ for $n = \num{10000}$, $p = -0.4$
      and $\sigma_0 = 0.3$}
    \label{tab:std-est-old}
\begin{tabular}{l l *{2}{S[group-separator={}]}}
  \toprule
  $H_0$ & $\alpha_0$ & \multicolumn{1}{c}{$\widetilde{\sigma}_{\textrm{low}}$} & \multicolumn{1}{c}{$\widetilde{\alpha}_{\textrm{low}}$}
  \\
  \midrule
  0.2  & 0.4 & 0.0514493641263876 & 0.224914101004984
  \\ 
        & 0.6 & 0.0420188025000937 & 0.219536634184391
  \\ 
        & 0.8 & 0.0676397059695992 & 0.218078984201489
  \\ 
        & 1 & 0.103586858524759 & 0.254877281466671
  \\ 
        & 1.2 & 0.132183941727461 & 0.263207721998055
  \\ 
        & 1.4 & 0.0999547935785991 & 0.230117226738914
  \\ 
        & 1.6 & 0.0510482066036124 & 0.136158804728185
  \\ 
        & 1.8 & 0.0264360584125776 & 0.0765195902850774
  \\ 
  \midrule
  0.4  & 0.4 & 0.0679797669342322 & 0.212668392981231
  \\ 
        & 0.6 & 0.0741427667333707 & 0.219761011341943
  \\ 
        & 0.8 & 0.152544460131115 & 0.245944190765084
  \\ 
        & 1 & 0.203866870481647 & 0.363293986470934
  \\ 
        & 1.2 & 0.0527577012335206 & 0.0875290925641975
  \\ 
        & 1.4 & 0.0235884027431762 & 0.0450578242897994
  \\ 
        & 1.6 & 0.015141837911502 & 0.0307654522636953
  \\ 
        & 1.8 & 0.0112121970829042 & 0.022233242048084
  \\ 
  \midrule
  0.6 & 0.4 & 0.0734342245948659 & 0.210842136970737
  \\ 
        & 0.6 & 0.152216520788966 & 0.214174891581134
  \\ 
        & 0.8 & 0.178989796636568 & 0.261797746218754
  \\ 
        & 1 & 0.124486608576801 & 0.189574378626174
  \\ 
        & 1.2 & 0.0259325349814443 & 0.037051055660017
  \\ 
        & 1.4 & 0.0148637220482232 & 0.0282569151499648
  \\ 
        & 1.6 & 0.0106408008745295 & 0.0252498973506683
  \\ 
        & 1.8 & 0.00782173194464961 & 0.0203545382747977
  \\ 
  \midrule
  0.8  & 0.4 & 0.0852770396327795 & 0.204132605421412
  \\ 
        & 0.6 & 0.285563236286601 & 0.232943615269778
  \\ 
        & 0.8 & 0.239870506454797 & 0.266382185164036
  \\ 
        & 1 & 0.0659177342284909 & 0.0793312719446012
  \\ 
        & 1.2 & 0.0201079626048648 & 0.0306230833362863
  \\ 
        & 1.4 & 0.0118226860359186 & 0.0280918016281635
  \\ 
        & 1.6 & 0.00842885161694 & 0.0254647018225799
  \\ 
        & 1.8 & 0.006386068798862 & 0.0208089201594853
  \\
  \bottomrule
\end{tabular}
  \end{minipage}
\end{table}

\begin{table}[htpb!]
  \centering
  \caption{Failure rates for $n = \num{10000}$, $p = -0.4$, $k = 2$,
    $\nu = 0.1$ and $\sigma_0 = 0.3$}
  \label{tab:failure-rate-old}
  \begin{tabular}{l l *{2}{S[table-format=2.2, round-precision = 2]}}
    \toprule
    & & \multicolumn{2}{c}{Failure rate (\%)}
    \\
    \cmidrule{3-4}
    $H_0$ & $\alpha_0$ & \multicolumn{1}{c}{$(\widetilde{\sigma}_{\textrm{low}}, \widetilde{\alpha}_{\textrm{low}}, \widetilde{H}_{\textrm{low}})$} & \multicolumn{1}{c}{$\xi_n$}
    \\
    \midrule
    0.2 & 0.4 & 88.17 & 14.83\\
    & 0.6 & 84.08 & 0\\
    & 0.8 & 69.67 & 0\\
    & 1 & 38.58 & 0\\
    & 1.2 & 0.42 & 0\\
    & 1.4 & 3.25 & 0.17\\
    & 1.6 & 3.83 & 1\\
    & 1.8 & 1.58 & 2.25\\
    \midrule
    0.4 & 0.4 & 86.08 & 20.17\\
    & 0.6 & 76.75 & 0.17\\
    & 0.8 & 62.58 & 0\\
    & 1 & 19.17 & 0.18\\
    & 1.2 & 0 & 0\\
    & 1.4 & 0 & 0.33\\
    & 1.6 & 0 & 1.42\\
    & 1.8 & 0 & 7.17\\
    \midrule
    0.6 & 0.4 & 87 & 28.67\\
    & 0.6 & 72.67 & 0\\
    & 0.8 & 41.75 & 0\\
    & 1 & 0.08 & 0\\
    & 1.2 & 0 & 0\\
    & 1.4 & 0 & 0.08\\
    & 1.6 & 0 & 2.42\\
    & 1.8 & 0 & 10.25\\
    \midrule
    0.8 & 0.4 & 89.5 & 26.33\\
    & 0.6 & 69.33 & 0.33\\
    & 0.8 & 2.17 & 0\\
    & 1 & 0.08 & 0.25\\
    & 1.2 & 0 & 0\\
    & 1.4 & 0 & 0.17\\
    & 1.6 & 0 & 3.92\\
    & 1.8 & 0 & 11\\
    \bottomrule
  \end{tabular}
\end{table}

\clearpage

\bibliographystyle{plain}
\phantomsection
\addcontentsline{toc}{section}{\refname}
\bibliography{References}

\end{document}